\documentclass[12pt]{amsart}
\usepackage{amsmath,amssymb, amscd, graphicx}

\makeatletter
\@addtoreset{equation}{section}
\makeatother
\usepackage{enumerate}

\def\ii{{\sqrt{-1}}}

\def\ee{\mathrm e}

\def\res{\mathrm {res}}

\def\sn{\mathrm {sn}}
\def\cn{\mathrm {cn}}
\def\dn{\mathrm {dn}}
\def\al{\mathrm {al}}

\def\Sym{{S}}

\def\nuII{{\nu^{II}}}
\def\nuI{{\nu^{I}}}

\def\CC{{\mathbb C}}
\def\ZZ{{\mathbb Z}}

\def\RR{{\mathbb R}}

\def\JJ{{\mathcal J}}

\def\Sp{{\mathrm{Sp}}}

\def\WW{{\mathcal{W}}}

\def\hzeta{\hat{\zeta}}

\newtheorem{theorem}{Theorem}[section]
\newtheorem{definition}[theorem]{Definition}

\newtheorem{proposition}[theorem]{Proposition}
\newtheorem{corollary}[theorem]{Corollary}
\newtheorem{remark}[theorem]{Remark}
\newtheorem{lemma}[theorem]{Lemma}

\def\dfrac#1#2{{\displaystyle\frac{#1}{#2}}}

\def\book#1{\rm{#1}, }
\def\paper#1{\textit{#1}, }
\def\jour#1{\rm{#1}, }
\def\yr#1{({\rm{#1}) }}
\def\vol#1{\textbf{#1}}
\def\pages#1{\rm{#1}}

\def\publaddr#1{\rm{#1}, }
\def\publ#1{\rm{#1}, }
\def\by#1{{\rm{#1}, }}

\def\eds{\rm{eds.}}

\begin{document}

\title{The $\al$ function 
of a cyclic trigonal curve of genus three}

\author{Shigeki Matsutani and Emma Previato}

\maketitle

\begin{abstract}
A cyclic trigonal curve of genus three is 
a $\mathbb{Z}_3$ Galois cover of $\mathbb{P}^1$,
therefore can be written as a  smooth plane curve with equation
$y^3 = f(x) =(x - b_1) (x - b_2) (x - b_3) (x - b_4)$.
Following Weierstrass
for the hyperelliptic case, we define an ``$\al$'' function
for this curve and $\al^{(c)}_r$, $c=0,1,2$, for each one of
three particular  covers of the Jacobian of the curve,
and $r=1,2,3,4$ for a finite branchpoint $(b_r,0)$.
 This 
 generalization of the Jacobi $\sn$, $\cn$, $\dn$ functions
 satisfies the relation:
$$
\sum_{r=1}^4 \frac{\prod_{c=0}^2\al_r^{(c)}(u)}{f'(b_r)} = 1
$$
which generalizes $\sn^2u + \cn^2u = 1$.
We also show that this can be viewed as a special case of the Frobenius theta identity.
\end{abstract}



\section{Introduction}

Jacobi's $\sn$, $\cn$, $\dn$ functions and  Weierstrass'
$\wp$ and $\sigma$ functions are closely connected with the coordinates
of the elliptic curve embedded in the affine plane.
The hyperelliptic  analog
of the Jacobi $\sn$, $\cn$, $\dn$ functions
was proposed by Weierstrass, who denoted it  ``$\al$'' in honor of Abel
  \cite{Wei}.
Solutions of completely integrable Hamiltonian
systems which linearize on a hyperelliptic Jacobian, 
such as the Neumann system
and the sine-Gordon equation, were produced using 
the $\al$ functions as phase-space coordinates \cite[Vol. II]{Mum}, \cite{Ma3}.

In this article we  extend the $\al$ function to a trigonal
curve by using Kleinian sigma functions \cite{Kl1, BEL1, EEL};
a possible application will be analogous
expressions for the solution of
the generalized Neumann  system studied by
 Schilling \cite{S} and 
 Adams, Harnad and  Previato \cite{AHP}, among others.
In the present work however the emphasis is on
the definition, and the algebraic constraints satisfied
by the cyclic $\al$ function, which in principle can be
generalized to any $\mathbb{Z}_n$-curve.
Such beautiful algebraic relations for Abelian functions
occur often in the literature,
not necessarily just for genus one: in particular the
article \cite{lindqvistpeetre} produces elementary proofs
(by substitution in the Abelian integrals) 
of generalized Ones and Twos, as large classes
of identities for inverses of Abelian integrals are known in Sweden.
It may be possible that our identities have like elementary proofs.

We work with smooth complete curves over the complex numbers,
namely compact Riemann surfaces.
For  a hyperelliptic
curve $C_g$: $y^2 = \prod_{i=1}^{2g+1} (x - b_i)$  of genus $g$,
we denote the Jacobian by $\JJ_g$ and the vector
in $\CC^g$ given by  integration\footnote{The ambiguity
due to  path of integration does not affect the formulas and is ignored
throughout.} 
from the base point
$\infty$ to the branch point $(x, y)=(b_a, 0)$ 
by
$\omega_a$.
A hyperelliptic $\al$ function is defined as:
$$
	\al_r(u) = \gamma_r''\frac{\ee^{-^t u \eta' \omega^{\prime -1}\omega_r}
                 \sigma(u+\omega_r)}{\sigma(u)},
$$
where $\sigma$ is  Klein's hyperelliptic $\sigma$ function 
\cite{Ba2} and the remaining symbols are defined in the Appendix.
If  for a point $u$ in the Jacobian $\JJ_g$, 
we choose any preimage under  the Abel map 
in the $g$-th symmetric product $S^g X_g$
and denote it
simply by 
$(x_i, y_i)_{i=1, \ldots, g}$ 
(meaning an unordered $g$-tuple),
then we can  give an algebraic expression of $\al_r(u)$
\begin{equation}
	\al_r(u) = \sqrt{F(b_r)}, \quad
        F(x) = \prod_{j=1}^g (x - x_j).
\label{eq:al_F}
\end{equation}
In order to fix the sign of the square root,
following Baker, we define the $\al$ function
on the  $g$-th symmetric product $S^g \hat C_{2g+1}$
where $\hat C_{2g+1}$ is a double cover of $C_g$
(see Appendix for details).
Weierstrass defined the $\al$ function using these ideas as well as 
an expression in terms of theta functions which he calls
$\mathrm{Al}$, using an analog of the elliptic sigma function,
a precursor of the Kleinian sigma function
 \cite{Wei, Kl1}\footnote{The letters
$\mathrm{al}$ and $\mathrm{Al}$ were used by Weierstrass in honor of  Abel.}.
Weierstrass investigated the $\al$ function  to construct
his version of the  sigma function for hyperelliptic curves, in terms of the
 affine coordinates of $S^g C_g$.
In the calculation \cite[p. 296]{Wei},
the sine-Gordon equation plays an
important role \cite{Ma1, Ma2}. Indeed,
the hyperelliptic $\al$ functions satisfy
the ellipsoidal relations:
\begin{equation}
	\sum_{r=1}^{g+1} c_{r} \al_r(u)^2 = 1,
\label{eq:sum_al2}
\end{equation}
where $c_{r}$ is a  constant that depends on
the branch points $b_a$'s.
This is a consequence 
of the Frobenius theta formula \cite[Ch.~III, Corollary 7.5]{Mum}:
\begin{equation}
\sum_{r=1}^{g+1}
c_{2g+1,r} \frac{\sigma(u+\omega_r)^2}{\sigma(u)^2} = 1,
\label{eq:sum_al2A}
\end{equation}
which gives the homogeneous relations in $\mathbb{P}^{2g+1}$
$$
c_{r,r}\sigma(u+\omega_{r})^2 +
\sum_{r'=1}^{g} c_{r,r'}\sigma(u+\omega_{r'})^2 =\sigma(u)^2, \quad
r=g+1, \ldots, 2g,
$$
where $c_{r,r'}$ is also a certain constant related to the $b_a$'s.
These $\al$-functions and the relation (\ref{eq:sum_al2}) are
a generalization of the Jacobi elliptic
functions and their relations,
$$
	\sn(v) = \sqrt{e_1-e_3} 
 \frac{\ee^{\eta_3 u} \sigma(\omega_3)\sigma(u)} {\sigma(u+\omega_3)}, \quad
	\cn(v) = 
 \frac{\ee^{(\eta_3-\eta_1) u} 
\sigma(\omega_3)\sigma(u+\omega_1)}
 {\sigma(\omega_1)\sigma(u+\omega_3)}, \quad
$$
$$
         \mbox{and} \quad
	\sn^2(v) + \cn^2(v) = 1,
$$
where $\sigma$ is the Weierstrass sigma function and
$v =u / \sqrt{e_1-e_3}$. 
Since $\sn(v) $ is proportional to $1/ \sqrt{\wp(u)-e_3}$,
the domain of $\sn(v)$ is a double cover of  $\JJ_1$ where $\wp(u)$
is defined.


\bigskip

Recently, further identities for
 the sigma function over a cyclic trigonal curve $X$:
$y^3 = f(x) = (x - b_1)(x- b_2)(x- b_3)(x- b_4)$ become available
\cite{EEL, EEMOP1, EEMOP2, O}. By using these results and
the $\mathbb Z_3$-symmetry of the curve, we  
define the trigonal {\lq\lq}$\al${\rq\rq} function and 
investigate its properties.
Again, to resolve a $\mathbb{Z}_3$-ambiguity, we will define a certain triple
 cover of the curve. For simplicity, in fact, we introduce the universal cover
of     $X$, which in turns admits a continuous
map to any cover of $X$, and we
use it to define an extended
 Abel map. Although the universal cover is not algebraic,
in fact unlike for $g=1$ it is the open unit disc, the values 
we get can also be computed algebraically using a triple cover of $X$;
we will introduce three finite covers of the Jacobian of $X$
labelled by $c=0,1,2$. Then,
the first of our main theorems  is the following:
\begin{theorem}
\begin{gather}
\al_a^{(c)}(u) := \frac{\ee^{-^t u\varphi_{a;c}}
            \sigma(u+\zeta_3^c\omega_a)}
            {\sigma_{33}(\zeta_3^c \omega_a)\sigma(u)}, \quad
\al_a^{(c)}(u) =- \zeta_3^{c+\varepsilon}
\frac{A_a(P_1, P_2, P_3)}{\sqrt[3]{F_a(P_1, P_2, P_3)}},
\label{eq:ala1}
\end{gather}
where $\zeta_3$ is a primitive third root of unity,
$a=1,\ldots,4$ labels the 4 branchpoints,
$A_a$ and $F_a$ are 
meromorphic functions of 3 (unordered) points
  defined in section \ref{functions};
 the vector $\varphi_{a;c}$ and a
  $\ZZ$-valued function $\varepsilon$
 on the preimage
 of $u$ under the Abel
map will be introduced below.
\end{theorem}

We arrive at a generalization of
(\ref{eq:sum_al2A}) \cite[Ch.~IIIa, Corollary 7.5]{Mum}
and obtain the second main theorem of this article:
\begin{theorem} (A generalized Frobenius' theta formula)
\label{th:1-1}
We have 
$$
	\sum_{r=1}^4 \frac{ \prod_{c=0}^2\al_r^{(c)}(u)}{f'(b_r)} = 1,
\quad
\sum_{a=1}^4
\frac{\prod_{c=0}^2\sigma(u+\hzeta_3^c\omega_a)}{\sigma(u)^3/(2\sqrt{2})} = 1.
$$
\end{theorem}
In the course of the study, by choosing an appropriate constant
multiple of the
sigma function, we obtain the additional identity:
$$
\sigma_{33}(\omega_a) =\left(
\sqrt[3]{\frac{d f(x)}{ dx}}\Bigr|_{x=b_a}\right)^{-1}.
$$

We remark that the definition of the trigonal $\al$-function and 
its properties
might depend upon the conventions we employ, {\it e.g.}, the path 
of integration in
 the Abelian coordinates, unlike the
algebraic functions of the curve. 
However,
Theorem \ref{th:1-1} reflects  the $\mathbb{Z}_3$-symmetry
of the Abelian variety and (\ref{eq:ala1}) connects
the $\sigma$-functions and the affine coordinates of the curve,
as the Jacobi $\sn$, $\cn$ and $\dn$ functions do.

The contents of this  article are as follows:
Section two presents the geometry of a genus-3 cyclic trigonal
curve in the affine plane.
Sections three and four are devoted to 
the addition law on the Jacobian. Section five is 
about functions, $A$ and $F$,
associated with the trigonal $\al$-function
as in (\ref{eq:ala1}).
Sections six and seven give a brief introduction
of the sigma function and its addition structure. Section eight is devoted to
the definition of the $\al$ function and relates the sigma function to the
 $A$ and $F$ functions. In section nine, we prove the analog of the 
Frobenius theta identity. Section ten studies the domain
of the $\al$-function. In the Appendix, we review the hyperelliptic 
$\al$ function.

\section{genus-3 $\mathbb{Z}_3$ curves}\label{curves}

A  curve $X$ of genus three with Galois action by  $\mathbb{Z}_3$
at one point can be represented by an affine plane model:

$$
 y^3 = f(x), \quad f(x):=(x - b_1)(x- b_2)(x- b_3)(x- b_4),
$$
where $b_i$'s are distinct complex numbers. Let the  branch point
$(b_i, 0)$ be denoted by $B_i$ $(i=1, 2, 3, 4)$.
A basis for the  holomorphic one-forms over $X$ 
is given by
$$
	\nuI_1 = \frac{dx}{3 y^2}, \quad
	\nuI_2 = \frac{x dx}{3 y^2}, \quad
	\nuI_3 = \frac{ dx}{3 y}.
$$
For a fixed primitive third root of unity $\zeta_3$,
there is an action $\hzeta_3$ on $X$ and the space $H^0(X, K_X)$ 
of holomorphic forms ($K_X$ denotes the canonical divisor,
and $K_X$ is also used for the corresponding sheaf),
induced from  a Galois action on $X$:

\begin{gather}
     \hzeta_3 (x, y) =(x, \zeta_3 y), \quad
        \hzeta_3\begin{pmatrix} \nuI_1\\ \nuI_2\\ \nuI_3 \end{pmatrix}
     =\begin{pmatrix} \zeta_3\nuI_1\\ \zeta_3\nuI_2\\ \zeta_3^2\nuI_3 \end{pmatrix}.
\end{gather}

We choose a $\mathbb{Z}$-basis 
$   \alpha_i, \beta_j,$  $ (1\leqq i, j\leqq 3)$
of $H_1(X,\ZZ)$ with
intersection numbers
$[\alpha_i, \alpha_j]=0$, $[\beta_i, \beta_j]=0$ and
$[\alpha_i, \beta_j]=-[\beta_i, \alpha_j]=\delta_{i,j}$  
illustrated in Figure 1, cf. \cite{EEP, Wel}.
\begin{figure}
\begin{center}
\includegraphics[scale=0.5]{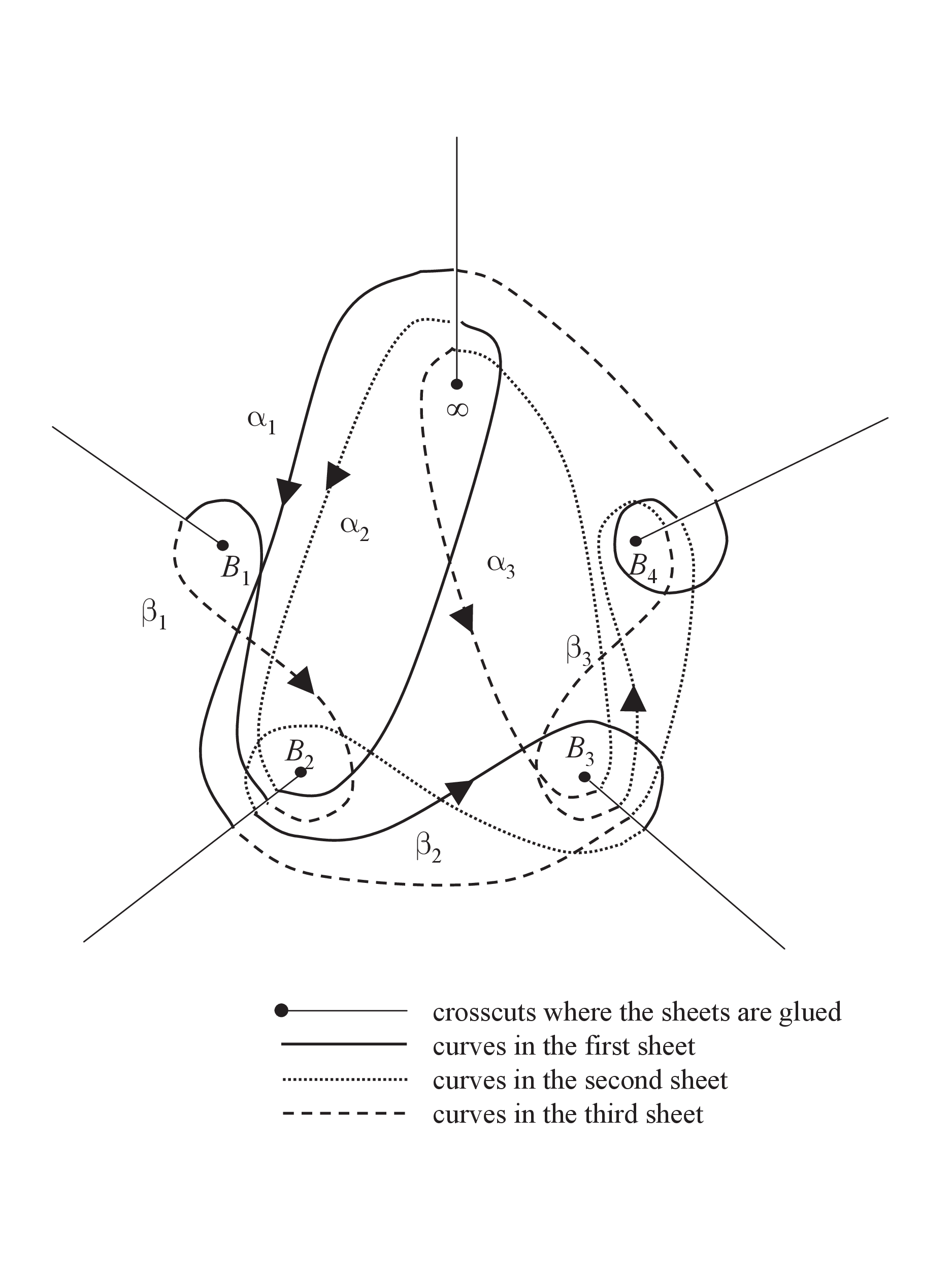}
\caption{Homology basis $\alpha_i$ and $\beta_i$, $i = 1, 2, 3$}
\label{fig:1}
\end{center}
\end{figure}


The half-period matrices for this basis are given by:
$$
	\omega' := (\omega_1', \omega_2', \omega_3'), \quad
	\omega'' := (\omega_1'', \omega_2'', \omega_3''), \quad
    {\Omega}:=\left(\begin{matrix} \omega' \\ \omega''
     \end{matrix}\right),
$$
where
$$
    \omega_a':=\frac{1}{2}\left[\int_{\alpha_{a}}\nuI\right], \quad
      \omega_a'':=\frac{1}{2}\left[\int_{\beta_{a}}\nuI\right], \quad
$$
with the convention 
that we go around a branchpoint along the paths drawn in Fig. 1, for example
we traverse $\beta_1$ around $B_1$ starting on sheet 1 and crossing over
to sheet 3.

A choice of
$\alpha$'s and $\beta$'s as in Fig. 1 yields certain relations, which
we'll use in computations even though strictly speaking they
hold modulo homotopy. Note that when acting by (powers of)
$\zeta_3$ we change sheet at a branchpoint: for example,
$\displaystyle{\int_\infty^{(b_a, 0)} \nuI - \hzeta_3 
\int_\infty^{(b_a, 0)} \nuI}=
\displaystyle{\int_\infty^{(b_a, 0)} \nuI +
\int^\infty_{(b_a, 0)} \hzeta_3\nuI}$.
\begin{proposition}
\label{lm:omega_a}
For
$
\displaystyle{\omega_a := \int_\infty^{(b_a, 0)} \nuI}
$, $(a = 1, 2, 3, 4)$ we have following relations:

\begin{enumerate} \item
We decompose the $\omega_a'$ and $\omega_a''$ in terms of
$\hzeta_3^c \omega_a$,
$$
	\omega'_1 = \frac{1}{2}
 \left((1-\hzeta_3^2) \omega_2 +\hzeta_3^2(1-\hzeta_3^2) \omega_3 
         +\hzeta_3(1-\hzeta_3) \omega_4\right), \quad
 \omega_a' = \frac{1}{2}\hzeta_3^{a-2} (\hzeta_3 - 1)\omega_a,
          \quad (a =  2, 3),
$$
$$
	 \omega_a'' = \frac{1}{2}
\hzeta_3^{a-2}(\hzeta_3 -1) (\omega_a -\omega_{a+1}),
          \quad (a = 1, 2, 3).
$$

 \item
$\hzeta_3^c \omega_a$ $(c=0, 1, a=1, 2, 3)$,
 are linear independent over $\ZZ$.
\end{enumerate}
\end{proposition}

\begin{proof}
(1) is directly obtained from Figure 1.
We obtain (2) through the following identities:
\begin{eqnarray}
	&(1 + \hzeta_3 + \hzeta_3^2) \omega_a = 0, \label{eq:Prop2.1}\\
	&(1-\hzeta_3^2) \omega_1 +\hzeta_3^2(1-\hzeta_3^2) \omega_2
	+\hzeta_3(1-\hzeta_3^2) \omega_3 +(1-\hzeta_3^2) \omega_4 = 0,
            \nonumber \\
	&(1-\hzeta_3) \omega_1 +\hzeta_3(1-\hzeta_3) \omega_2
	+\hzeta_3^2(1-\hzeta_3) \omega_3 +(1-\hzeta_3) \omega_4 = 0.
 \nonumber
\end{eqnarray}
The first identity comes from the fact
$$
  \int_{\infty}^{(x_1, y_1)} 
    \begin{pmatrix} \nuI_1\\ \nuI_2\\ \nuI_3 \end{pmatrix}
  + \int_{\infty}^{(x_1, \zeta_3 y_1)} 
    \begin{pmatrix}\zeta_3^2 \nuI_1\\ \zeta_3^2 \nuI_2\\ \zeta_3 \nuI_3\end{pmatrix}
  +  \int_{\infty}^{(x_1, \zeta_3^2 y_1)} 
    \begin{pmatrix}\zeta_3 \nuI_1\\ \zeta_3 \nuI_2\\ \zeta_3^2 \nuI_3 \end{pmatrix}
= 0,
$$
because $(1 + \zeta_3 + \zeta_3^2) = 0$. 
The others are obtained by integrating along paths which, as seen in 
Figure 1, are homotopic to a point.
\end{proof}


Let $\Lambda$ be the lattice in $\CC^3$
generated by $2\omega'$ and $2\omega''$.
The  universal covering space of $X$
is homeomorphic to the space of equivalence classes
(up to homotopy) of paths in $X$ which begin at some fixed point $P$;
for simplicity,  we use the space of paths $\Gamma_P X$ because
all the functions we define are independent of homotopy.
The map $\kappa_P: \Gamma_P X \to X$ such that
for a path $\Gamma_{Q,P}$ from $P$ to $Q$, $\kappa_P(\Gamma_{Q,P})=Q$
defines a fiber structure on $\Gamma_P X$.
The path $\Gamma_{Q,P}$ can be decomposed into
$\Gamma_{Q,P} = \sum_{i=1}^g (n_i \alpha_i+n'_i \beta_i) + \Gamma'_{Q,P}$
up to homology,
where $\Gamma'_{Q,P}$ is a simple curve from $P$ to $Q$ without any 
loops in $X$, so that the integral of a holomorphic differential
on $\Gamma_{Q,P}$ and $\Gamma'_{Q,P}$ is the same, modulo periods.

We extend the Abel map $w$ and define, using the same letter,
a map from $\Gamma_\infty X$
 to $\CC^3$:
$$
 w(\Gamma_{(x,y), \infty})
     := \int_{ \Gamma_{(x,y), \infty}} \nuI, \quad \nuI:=
        \begin{pmatrix} \nuI_1\\ \nuI_2\\ \nuI_3 \end{pmatrix}.
$$
We simply write $w(x, y)$ for $(x,y) \in X$.
We write  $(x_i, y_i)_{i=1, 2, \ldots, k}$ to indicate
(by slightly abusing notation)
 an element of the symmetric product $S^k \Gamma_\infty X$, and
we extend the Abel map by 
$$
	w((x_i, y_i)_{i=1, 2, \ldots, k}) = 
	w(\Gamma_{(x_i, y_i),\infty}|_{i=1, 2, \ldots, k}) := 
	  \sum_{i=1}^k  w(\Gamma_{(x_i,y_i), \infty})
	  =\sum_{i=1}^k  w(x_i,y_i).
$$
The map $w$ is surjective for $k\ge 3$ (Abel-Jacobi theorem).
We denote by $\mathcal{J}$ the Jacobian of $X$
and by $\kappa$ the natural projection defined by the lattice $\Lambda$,
$$
	\kappa: \CC^3 \to \JJ = \CC^3/\Lambda.
$$
The strata of $\JJ$,  $\WW_k:= \kappa w(S^k \Gamma_\infty X)$ $(k\ge 1)$,
are the same as the sets $\kappa w(S^k X)$, abbreviating
$w\circ \kappa_\infty$ as $w$.


The $\mathbb{Z}_3$ action on $X$ 
 induces an action on the Jacobian $\hzeta_3: \JJ \to \JJ$
such that $\hzeta_3^3 = id$ and equivariantly, a map
on a  preimage $(P_1, P_2, P_3)\in S^3 X$ 
of the Abel map, given by
$(\hzeta_3 P_1, \hzeta_3 P_2, \hzeta_3 P_3)$, indeed
	$\hzeta_3 \Lambda = \Lambda$ as seen by the action on
the paths of integration.

\begin{proposition}
\label{prop:hcab}
$\{3\zeta_3^c \omega_a\}_{a = 1, 2, 3, 4, c = 0, 1, 2}$ is 
a subset of $\Lambda$.
For every $3\hzeta_3^c \omega_a$,
there are integers $h^{(c)\prime}_{a, b}$ and $h^{(c)\prime\prime}_{a, b}$
$(a=1,2,3,4,\ b = 1, 2, 3,\ c=0, 1, 2)$ such that
\begin{equation}
\hzeta_3^c \omega_a = 
\frac{2}{3}\sum_{b=1}^3 (h^{(c)\prime}_{a, b}\omega_b' + 
      h^{(c)\prime\prime}_{a, b}\omega_b'').
\label{eq:invL}
\end{equation}
\end{proposition}

Here we point out that even though the values 
$h^{(c)\prime}_{a, b}$ and $h^{(c)\prime\prime}_{a, b}$ depend 
upon the choice of the homology basis $\alpha_a$ and 
$\beta_a$, the above fact that
$3\zeta_3^c\omega_a$ is a point in $\Lambda$
does not, and thus the $\al$-function
defined below is invariant under
the action of $\Sp(6, \ZZ)$ on a 3-vector.

\begin{proof}
 From $(1-\zeta_3) (1-\zeta_3^2) =3$, we have
$$
\hzeta_3^{2a-1}(\hzeta_3^2-1)\omega_a' = \frac{3}{2}\omega_a, \quad (a= 2, 3),
\quad
\hzeta_3^{2a-2}(\hzeta_3^2 - 1)\omega_a'' 
= \frac{3}{2}(\omega_a- \omega_{a+1}), \quad (a = 1, 2, 3).
$$
Then we have the relations:
$\omega_a = \frac{2}{3}\hzeta_3^{2a-1}(\hzeta_3^2-1)\omega_a'$ 
$(a= 2, 3)$, 
$\omega_4 = \frac{2}{3}\hzeta_3^{2a-2}(\hzeta_3^2-1)(\hzeta_3\omega_3'-
\omega_3'')$, 
and
$\omega_1 = \frac{2}{3}\hzeta_3^{2a-2}(\hzeta_3^2-1)(\omega_1''+ \hzeta_3
\omega_2')$. 
Since 
 the $\ZZ$-module $\Lambda$  is invariant under the action of $\hzeta_3$,
then $2 \hzeta_3 \omega_a' \in \Lambda$ and 
$2 \hzeta_3 \omega_a'' \in \Lambda$, hence there are integers
$p^{(c)\prime}_{a, b}$, $p^{(c)\prime\prime}_{a, b}$,
$q^{(c)\prime}_{a, b}$ and $q^{(c)\prime\prime}_{a, b}$, such that
$$ 
\hzeta_3^c \omega_a' =\sum_{b=1}^3 (p^{(c)\prime}_{a, b}\omega_b' + 
      p^{(c)\prime\prime}_{a, b}\omega_b''),
\quad
\hzeta_3^c \omega_a'' =\sum_{b=1}^3 (q^{(c)\prime}_{a, b}\omega_b' + 
      q^{(c)\prime\prime}_{a, b}\omega_b''),
$$
which shows the statement.
\end{proof}

\begin{figure}
\begin{center}
\includegraphics[scale=0.7]{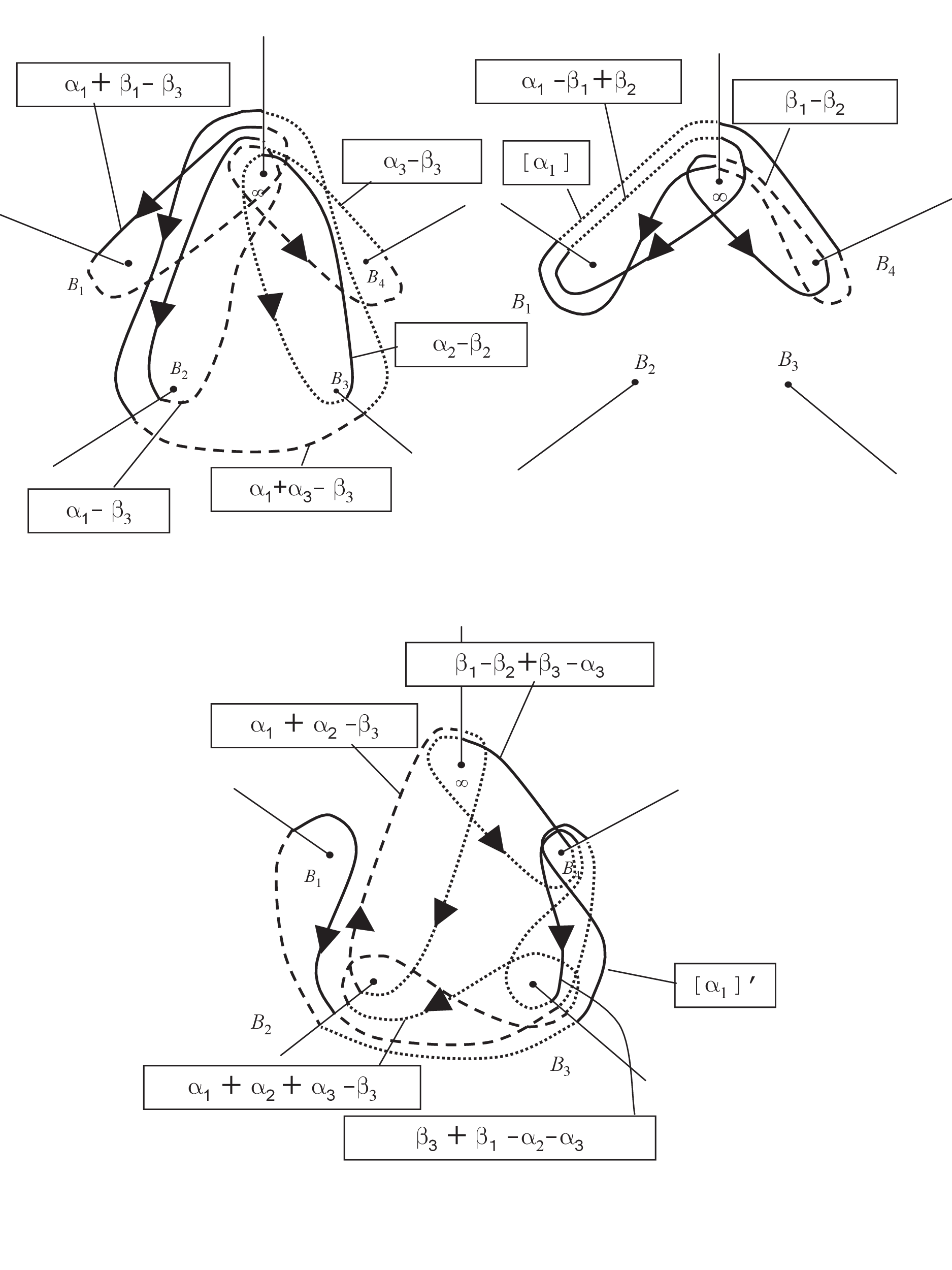}
\caption{Relations among $\alpha_i$ and $\beta_i$}
\label{fig:2}
\end{center}
\end{figure}

\begin{remark}\label{rmk:w''}
{\rm{We add 
 oriented loops $\gamma$ and $\gamma'$ up to equivalence, in the homology
group of the curve:
$\int_{\gamma+\gamma'} \nuI = \int_{\gamma} \nuI + \int_{\gamma'} \nuI$.
Using the relations in the proof of Proposition \ref{lm:omega_a},
we can compute  the $\omega$'s along the paths  in Figure 2,
where $[\alpha_1]$ 
 is an equivalence class
modulo (\ref{eq:Prop2.1}), though we routinely abuse notation and write
simply $\alpha_1$.
Using Figure 2, one checks identities such as:
$$
   3\omega_1 = 2 \omega_1' + \omega_2'' - \omega_3''.
$$
}}
\end{remark}


\begin{remark} \label{rmk:holo}
{\rm{
We summarize the local  behavior of the holomorphic one-forms for
use below. In a $t$-series expansion,
$d_\ge(t^\ell)$ denotes a  term of 
 order 
 greater than or equal to $\ell$.
\begin{enumerate}
\item
At the point $\infty$, we choose a local parameter $t_\infty $ so that
 $t_\infty^3 = 1/x$ and
 $y = \dfrac{1}{t_\infty^4} (1 + d_{\ge}(t_\infty))$;
the holomorphic one-forms are expanded as,
$$
	\nuI_1 = - t_\infty^4(1 + d_{\ge}(t_\infty)) dt_\infty, \quad
	\nuI_2 = - t_\infty^3(1 + d_{\ge}(t_\infty)) dt_\infty, \quad
	\nuI_3 = - (1 + d_{\ge}(t_\infty)) dt_\infty.
$$

\item
At $(b_a, 0)$, a local parameter $t_a$ is chosen so that $t_a^3 = (x-b_a)$.
Then we have
$$
    y = t_a  C_{a} (1 + d_{\ge}(t_a^2)), \quad
	d x = 3 t_a^2 dt_a,
$$
where  $C_{a}:= 
\displaystyle{\left.\sqrt[3]{\frac{d f(x)}{d x}}\right|_{x = b_a}}$.
The holomorphic one-forms are written as,
$$
	\nuI_1 = 
    \frac{dt_a}{C_{a}^2} (1 + d_{\ge}(t_a)),
    \quad \nuI_3 = 
    \frac{t_a dt_a}{C_{a}} (1 + d_{\ge}(t_a)),
    \quad \nuI_2 - b_1 \nuI_1 = 
    \frac{t_a^3 dt_a}{C_{a}^2} (1 + d_{\ge}(t_a)).
$$
The local chart is a triple covering of the curve projected to the $x$-axis,
so there is a natural action 
$\hat\iota^{(a)}_{\zeta_3} : t_a \to \zeta_3 t_a$.
Since $y$ is also a local parameter at a branch point $B_a$, locally
we can identify the action $\hat\iota^{(a)}_{\zeta_3}$  with $\hat \zeta_3$. 
\end{enumerate}
}}
\end{remark}


 The following 
meromorphic functions $\{\phi_i\}$ on $X$ belong to 
the  ring $R := \CC[x, y] / (y^3 - \prod(x - b_r))$,
\begin{gather}
	\phi_i = \left\{
        \begin{matrix} 
     x & \mbox{for } i = 3 \\
     y & \mbox{for } i = 4 \\
     x^{(i-3)/3 + 2} & \mbox{for } i > 3, i \equiv 0  \mbox{ modulo } 3,\\
     x^{(i-3)/3 + 1}y & \mbox{for } i > 3, i \equiv 1 \mbox{ modulo } 3,\\
     x^{(i-3)/3 }y^2 & \mbox{for } i > 3, i \equiv 2  \mbox{ modulo } 3.
        \end{matrix} \right.
\end{gather}
In particular, $\phi_n$ has a
pole of order  $N(n)$ at $\infty$,
with:
\begin{gather}
       N(0) = 0, \quad
       N(1) = 3, \quad
       N(2) = 4, \quad
       N(3) = 6, \quad
       N(n) = n + 3, \quad \mbox{for } n > 3.
\end{gather}

Lastly, we identify the Jacobian with $\mathrm{Pic}^0(X)$
by choosing $\infty$ as the base point, so we embed $X$ in 
$\mathcal{J}$ by sending 
 a point $P\in X$ to
 the sheaf associated to the divisor  $P - \infty$
(up to linear equivalence).

\section{Addition law. I}

The additive inverse
in $\CC^3$ and $\JJ$ corresponds to a bijection from
 $S^3 X$ to itself, which depends on our
choice of base point. We note that in the hyperelliptic case
this corresponds simply to  the hyperelliptic involution $(x, y) \to (x, -y)$,
in the (cyclic) trigonal  case
we need a modification as in \cite[Lemma 2.6]{MP}.

We now give an explicit realization of 
the Serre involution on $\mathrm{Pic}^{(g-1)}$:
$$
\mathcal{L} \to K_X \mathcal{L}^{-1},
$$
which will be used in the development below.

The following proposition \cite[Lemma 2.6]{MP} holds:
\begin{proposition} \label{prop:addition}
For a positive integer $n$, there
is a natural inclusion satisfying
$$
[-1] :  \WW_n \to  \WW_{N(n) - n}, \quad
(u \mapsto -u).
$$
\end{proposition}

We construct an algorithm to give the $[-1]$-action 
explicitly in order to  define the trigonal $\al$ function.

\begin{definition} \label{def:mul}
For $(P, (P_1, \cdots, P_n))$ $\in (X-\infty) \times S^n(X-\infty)$,
we define $\mu_n(P)$ by
\begin{gather*}
\mu_n(P): = 
\mu_n(P; P_1,  \ldots, P_n): = 
\lim_{P_i' \to P_i}
\dfrac{
\left|\begin{matrix}
1 &\phi_1(P_1') & \cdots  &\phi_{n}(P_1') \\
1 &\phi_1(P_2') & \cdots  &\phi_{n}(P_2') \\
\vdots & \vdots & \ddots& \vdots \\
1 &\phi_1(P_{n}')&\cdots  &\phi_{n}(P_{n}')\\
1 &\phi_1(P)&\cdots   &\phi_{n}(P)
\end{matrix}\right|}
{\left|\begin{matrix}
1 &\phi_1(P_1') & \cdots &\phi_{n-1}(P_1') \\
1 &\phi_1(P_2') & \cdots &\phi_{n-1}(P_2') \\
\vdots & \vdots & \ddots& \vdots\\
1 &\phi_1(P_{n}')&\cdots &\phi_{n-1}(P_{n}')
\end{matrix}\right|},
\end{gather*}
and since the $\phi$'s are algebraic functions on $X$,
we extend the domain to $X\times S^n(X)$,
allowing for poles.
\end{definition}

More details on this function are given in \cite{MP},
 where it is shown that it can be viewed as a generalization of
the $F$ in (\ref{eq:al_F})  or  $U$
 in the triple of polynomials that  Mumford 
calls $(U,V,W)$ in the hyperelliptic case \cite[Ch.~IIIa]{Mum}.

\bigskip
In the following Lemma, we show that $\mu_n(P)$ is associated with
 the addition structure  on the divisor group of
$X$ from a classical viewpoint.
Let $S^n_1(X)$ be defined as
$$
\Bigr\{(P_1, \ldots, P_n) \in S^n(X) 
           \ | \ \exists \{{i_1}, \ldots, {i_k}\} \subset \{1, \ldots, n\}
 \mbox{ such that }
           \mu_k(P_{i_1}, \ldots, P_{i_k}) = 0\Bigr\}.
$$

For $(P_i)_{i=1,\ldots, n}\in S^n(X -\infty) \setminus S^n_1(X-\infty) $,
$\mu_n(P; P_1, \ldots, P_n)$ has the following properties:
\begin{enumerate}
\item It is monic,
\item At each $P=P_i$, it has a simple zero, 
\item $\mu_n(P)$ has a pole at $\infty$ of order $N(n)$, and
\item $\mu_n(P)$ has $(N(n) - n)$ zeros 
aside from the $P_i$'s $(i=1, \ldots, n)$.
\end{enumerate}

\begin{lemma}\label{prop:2theta2}
Let $n$ be a positive integer.
For $(P_i)_{i=1,\ldots, n}\in S^n(X -\infty) $,
 $\mu_n$ is consistent with the following diagram, where $[-1]_n$:
$S^n(X -\infty)  \to S^{N(n) - n}(X)$ such that 
$$
\begin{CD}
S^n(X -\infty) @>{[-1]_n}>> S^{N(n) - n}(X) \\
  @V{w}VV @V{w}VV \\
\WW_n @>{[-1]}>> \WW_{N(n) - n} \\
\end{CD},
$$
{\it{i.e.}}, 
$(P_i)_{i=1,\ldots, n}\in S^n(X -\infty) $ corresponds to
an element $(Q_i)_{i=1,\ldots, N(n)-n}\in S^{N(n)-n}(X)$,
such that
$$
\sum_{i=1}^{n} P_i - n \infty 
\sim - \sum_{i=1}^{N(n)-n} Q_i  + (N(n)-n) \infty .
$$
\end{lemma}

\bigskip

\section{Addition law: examples}
\label{sec:-1}

In this section, we compute  $[-1]_n$ $(n=1, 2, 3)$ explicitly,
based on Lemma \ref{prop:2theta2}, to demonstrate the significance of $\mu_n$.

\subsection{$[-1]_1P$:}
To find $[-1]_1(x_1, y_1)$, recall that N(1)=3, so
we seek a divisor of degree two.
For the divisor of a meromorphic function of $(x, y)$,
\begin{gather*}
\left|\begin{matrix}
1 & x_1 \\
1 & x \\
\end{matrix}\right| =0
\end{gather*}
implies that $x=x_1.$
This means that 
$$ 
   (x_1, y_1) + (x_1,\zeta_3 y_1) + (x_1,\zeta_3^2 y_1) -3\infty \sim 0. 
$$
In other words, as sheaves,
\begin{equation}
[-1]_1 (x_1, y_1) =  (x_1,\zeta_3 y_1) + (x_1,\zeta_3^2 y_1).  
\label{eq:[-1]}
\end{equation}
This means 
$$
  \int_{\infty}^{(x_1, y_1)} 
    \begin{pmatrix} \nuI_1\\ \nuI_2\\ \nuI_3 \end{pmatrix}
  + \int_{\infty}^{(x_1, \zeta_3 y_1)} 
    \begin{pmatrix}\zeta_3^2 \nuI_1\\ \zeta_3^2 \nuI_2\\ \zeta_3 \nuI_3\end{pmatrix}
  +  \int_{\infty}^{(x_1, \zeta_3^2 y_1)} 
    \begin{pmatrix}\zeta_3 \nuI_1\\ \zeta_3 \nuI_2\\ \zeta_3^2 \nuI_3 \end{pmatrix}
\equiv 0, \quad \mbox{modulo } \Lambda.
$$
%

\subsection{$[-1]_2(P_1 + P_2)$:}

Let $Q_a = (x'_a, y'_a)$ ($a=1,2$) be a  solution of $\mu_3(P, P_1, P_2)$
different from   $P_a$, i.e.,
$$
	w(Q_1 + Q_2) = -w(P_1 + P_2), 
	\quad{i.e.},
        (Q_1 , Q_2) = [-1]_2(P_1 , P_2).
$$
Since
\begin{gather}
\left| \begin{matrix}
1 & x_1 & y_1 \\
1 & x_2 & y_2 \\
1 & x_a'   & y_a'   \\
\end{matrix} \right| =
\left| \begin{matrix}
 x_1 & y_1 \\
 x_2 & y_2 \\
\end{matrix} \right| -
\left| \begin{matrix}
1 &  y_1 \\
1 &  y_2 \\
\end{matrix} \right| x_a' +
\left| \begin{matrix}
1 & x_1  \\
1 & x_2  \\
\end{matrix} \right| y_a' = 0,
\label{eq:4.2}
\end{gather}
then direct computations provide the following lemma:
\begin{lemma} \label{lm:-1two}
For the points obeying (\ref{eq:4.2})
the following relations hold:
\begin{gather*}
\begin{split}
&\left| \begin{matrix}
 x_1 & y_1 \\
 x_2 & y_2 \\
\end{matrix} \right| 
\left| \begin{matrix}
1 &  y_1' \\
1 &  y_2' \\
\end{matrix} \right| =
\left| \begin{matrix}
 x_1' & y_1' \\
 x_2' & y_2' \\
\end{matrix} \right| 
\left| \begin{matrix}
1 &  y_1 \\
1 &  y_2 \\
\end{matrix} \right|, \quad 
\left| \begin{matrix}
 x_1 & y_1 \\
 x_2 & y_2 \\
\end{matrix} \right| 
\left| \begin{matrix}
1 &  x_1' \\
1 &  x_2' \\
\end{matrix} \right| =
\left| \begin{matrix}
 x_1' & y_1' \\
 x_2' & y_2' \\
\end{matrix} \right| 
\left| \begin{matrix}
1 &  x_1 \\
1 &  x_2 \\
\end{matrix} \right|, \\
& \left| \begin{matrix}
1 &  y_1 \\
1 &  y_2 \\
\end{matrix} \right| 
\left| \begin{matrix}
1 &  x_1' \\
1 &  x_2' \\
\end{matrix} \right| =
\left| \begin{matrix}
1 &  y_1' \\
1 &  y_2' \\
\end{matrix} \right| 
\left| \begin{matrix}
1 &  x_1 \\
1 &  x_2 \\
\end{matrix} \right|.
\end{split}
\end{gather*}
\end{lemma}

This shows that if $P_1$ and $P_2$ are
generic in the sense that none of the determinants in
Lemma \ref{lm:-1two} vanishes, then
$Q_1$ and $Q_2$ have the same property.

\subsection{$[-1]_3(P_1+P_2+P_3)$}
Let 
$$
(Q_1, Q_2, Q_3) = [-1]_3(P_1, P_2, P_3).
$$
In other words $Q_{a} = (x'_a, y'_a)$ ($a = 1, 2, 3$) are
 solutions of $\mu_3(P, P_1, P_2, P_3)$
which differ from $P_b$ ($b = 1, 2, 3$).
\begin{gather}
\left| \begin{matrix}
1 & x_1 & y_1 & x_1^2 \\
1 & x_2 & y_2 & x_2^2 \\
1 & x_3 & y_3 & x_3^2 \\
1 & x'_a   & y'_a  & {x'}_a^2   \\
\end{matrix} \right| = 0.
\label{eq:III}
\end{gather}

Similar to Lemma \ref{lm:-1two}, we have the following
result:
\begin{lemma} \label{lm:-1three}
For $\{ i_1, i_2, i_3\} , \{ j_1, j_2, j_3\}
\subset\{ 0, 1, 2,3\}$, the following relation holds
\begin{gather*}
\begin{split}
&
\left|\begin{matrix}
\phi_{i_1}(P_1) &\phi_{i_2}(P_1) &\phi_{i_3}(P_1)\\
\phi_{i_1}(P_2) &\phi_{i_2}(P_2) &\phi_{i_3}(P_2)\\
\phi_{i_1}(P_3) &\phi_{i_2}(P_3) &\phi_{i_3}(P_3)\\
\end{matrix}\right|
\left|\begin{matrix}
\phi_{j_1}(Q_1) &\phi_{j_2}(Q_1) &\phi_{j_3}(Q_1)\\
\phi_{j_1}(Q_2) &\phi_{j_2}(Q_2) &\phi_{j_3}(Q_2)\\
\phi_{j_1}(Q_3) &\phi_{j_2}(Q_3) &\phi_{j_3}(Q_3)\\
\end{matrix}\right|\\
=&\epsilon_{i_1, i_2, i_3, j_1, j_2, j_3}
\left|\begin{matrix}
\phi_{j_1}(P_1) &\phi_{j_2}(P_1) &\phi_{j_3}(P_1)\\
\phi_{j_1}(P_2) &\phi_{j_2}(P_2) &\phi_{j_3}(P_2)\\
\phi_{j_1}(P_3) &\phi_{j_2}(P_3) &\phi_{j_3}(P_3)\\
\end{matrix}\right|
\left|\begin{matrix}
\phi_{i_1}(Q_1) &\phi_{i_2}(Q_1) &\phi_{i_3}(Q_1)\\
\phi_{i_1}(Q_2) &\phi_{i_2}(Q_2) &\phi_{i_3}(Q_2)\\
\phi_{i_1}(Q_3) &\phi_{i_2}(Q_3) &\phi_{i_3}(Q_3)\\
\end{matrix}\right|,
\end{split}
\end{gather*}
where $\epsilon_{i_1, i_2, i_3, j_1, j_2, j_3}$ is an appropriate
sign.
\end{lemma}

\bigskip

\section{Functions $A_r$ and $F_r$}\label{functions}

In order to construct the trigonal $\al_r$ function of the curve $X$,
we introduce meromorphic functions $A_r$ and $F_r$.

On a hyperelliptic curve, as shown in (\ref{eq:al_F}),
the $\al$ function is 
alternatively defined by
$\al_r(u) := \sqrt{F(b_r)}$ up to a constant factor, 
where $(x, y) = (b_r, 0) = B_r$ is
a branch point of the curve.

In order to define the trigonal version of $\al$ function,
we also deal with the value of the function
$\mu$ in (\ref{eq:defAF}) at a  branch
point $B_r$.


\begin{definition}
For a branch point $B_a$ of $X$ and 
$P_i = (x_i, y_i)$ $(i = 1, 2, 3)$ determining a point of
$S^3 X$, we define the meromorphic functions:
\begin{gather}
\begin{split}
A_a(P_1, P_2, P_3)&:=
\mu(B_a; P_1, P_2, P_3) = \displaystyle{ \frac{
     \left| \begin{matrix}
     1 & x_1 & y_1 & {x_1}^2 \\
     1 & x_2 & y_2 & {x_2}^2 \\
     1 & x_3 & y_3 & {x_3}^2 \\
     1 & b_a & 0   & b_a^2 \\
     \end{matrix} \right|} {
     \left| \begin{matrix}
     1 & x_1 & y_1  \\
     1 & x_2 & y_2  \\
     1 & x_3 & y_3  \\
     \end{matrix} \right|}},\\
F_a(P_1, P_2, P_3)&:= 
(b_a - x_1) (b_a - x_2) (b_a - x_3).\\
\end{split}
\label{eq:defAF}
\end{gather}
\end{definition}


Let $\deg_P h$ denote the order of
zero or pole of a meromorphic function $h$
 at $P$. 

\begin{proposition}\label{prop:Azero}
$A_a$ and $F_a$ have the following zeros and poles:
\begin{enumerate}
\item
For generic points $P_1, P_2$ of $X$,
$$
 \deg_{(B_a, P_1, P_2)} A_a(P_1, P_2, P_3) = 1,  \quad
 \deg_{(B_a, P_1, P_2)} F_a(P_1, P_2, P_3) = 3.  \quad
$$

\item
For generic points $P_1, P_2$ of $X$,
$$
 \deg_{[-1]_3([-1]_2(P_1, P_2), B_a)} A_a = 1,  \quad
 \deg_{[-1]_3([-1]_2(P_1, P_2), B_a)} F_a = 0.  \quad
$$

\item
For generic points $P_1, P_2$ of $X$,
$$
 \deg_{(\infty, P_1, P_2)} A_a = -2,  \quad
 \deg_{(\infty, P_1, P_2)} F_a = -3.  \quad
$$

\item
For generic points $P_1$ of $X$,
$$
 \deg_{([-1](b_a, 0), P_1)} A_a = 3,  \quad
 \deg_{([-1](b_a, 0), P_1)} F_a = 6.  \quad
$$
\end{enumerate}
\end{proposition}

\begin{proof}

(1) follows from the definition.
(3) is obvious because for $P_3$ near $\infty$, we have
\begin{equation}
\begin{split}
	A_a(P_1, P_2, P_3)
    &= \frac{ \left|\begin{matrix}
               1 & x_2 & y_2\\
               1 & x_3 & y_3\\
               1 & b_2 & 0 \\
               \end{matrix}\right| x_1^2}{
               \left|\begin{matrix}
               1 & x_2 \\
               1 & x_3 \\
      \end{matrix}\right| y_1} (1 + d_{\ge}(t_\infty^2))
    = \frac{ \left|\begin{matrix}
               1 & x_2 & y_2\\
               1 & x_3 & y_3\\
               1 & b_2 & 0 \\
               \end{matrix}\right| }{
               \left|\begin{matrix}
               1 & x_2 \\
               1 & x_3 \\
      \end{matrix}\right|} \frac{1}{t_\infty^2}(1 + d_{\ge}(t_\infty^2)), \\
	F_a(P_1, P_2, P_3)
    &= \frac{1}{t_\infty^3} + d_{\ge}(t_\infty^2). \quad
\label{eq:AF3}
\end{split}
\end{equation}

To prove (2), we denote by
  $C_a:=\left.\dfrac{d y}{d t_a}\right|_{x = b_a}$,
we assume that $P_1$ and $P_2$ are generic points
and $P_3$ is close to $B_a$;
$P_3 = (x_3, y_3)$ behaves like 
$(b_a + t_a^3 + d_{\ge}(t_a^4), \zeta_3^i C_a t_a + d_{\ge}(t_a^2) )$.
Let
$$
(Q_1, Q_2) = [-1]_2(P_1, P_2), \quad
(Q_1',Q_2',Q_3') =[-1]_3 (Q_1, Q_2, P_3),
$$
where $Q_a = (x'_a, y'_a)$ and $Q'_a = (x''_a, y''_a)$, {\it i.e.},
\begin{gather*}
\left| \begin{matrix}
1 & x_1 & {y_1} \\
1 & x_2 & {y_2} \\
1 & x_c' & {y_c'} \\
\end{matrix} \right| = 0,
\quad
\left| \begin{matrix}
1 & x_1' & {y_1'} & {x_1'}^2 \\
1 & x_2' & {y_2'} & {x_2'}^2 \\
1 & x_3 & {y_3} & {x_3}^2 \\
1 & x_c'' & {y_c''} & {x_c''}^2 \\
\end{matrix} \right| = 0.
\end{gather*}

We  consider the expansion of $A_a$ and $F_a$.
First we look at $F_a$. When $P_3$ is equal to $B_a$,
$F_a$ becomes
\begin{gather}
\left| \begin{matrix}
1 & x_1'' & {x_1''}^2 & {x_1''}^3 \\
1 & x_2'' & {x_2''}^2 & {x_2''}^3 \\
1 & x_3'' & {x_3''}^2 & {x_3''}^3 \\
1 & b_a & b_a^2   & b_a^3 \\
\end{matrix} \right| \Bigr/
\left| \begin{matrix}
1 & x_1'' & {x_1''}^2  \\
1 & x_2'' & {x_2''}^2  \\
1 & x_3'' & {x_3''}^2  \\
\end{matrix} \right|,
\end{gather}
and if it vanished then
 one of the $Q'$'s would equal $(b_a + t^3, \zeta_3 C_a t)$ near $t=0$
but  $(b_a + t^3, \zeta_3 C_a t)$ does not satisfy (\ref{eq:4.2}).

We now compute $A_a$ as follows:
\begin{gather*}
\begin{split}
A_a( Q'_1, Q'_2, Q'_3)
&= \displaystyle{
\frac{
\left| \begin{matrix}
1 & x_1'' & y_1'' & {x_1''}^2 \\
1 & x_2'' & y_2'' & {x_2''}^2 \\
1 & x_3'' & y_3'' & {x_3''}^2 \\
1 & b_a & 0   & b_a^2 \\
\end{matrix} \right|}
{
\left| \begin{matrix}
1 & x_1'' & y_1''  \\
1 & x_2'' & y_2''  \\
1 & x_3'' & y_3''  \\
\end{matrix} \right|
}}.\\
\end{split}
\end{gather*}
Direct computation gives the numerator as:
\begin{gather}
\left| \begin{matrix}
 x_1'' & y_1'' & {x_1''}^2 \\
 x_2'' & y_2'' & {x_2''}^2 \\
 x_3'' & y_3'' & {x_3''}^2 \\
\end{matrix} \right| -
\left| \begin{matrix}
1 &  y_1'' & {x_1''}^2 \\
1 &  y_2'' & {x_2''}^2 \\
1 &  y_3'' & {x_3''}^2 \\
\end{matrix} \right| b_a -
\left| \begin{matrix}
1 & x_1'' & y_1''  \\
1 & x_2'' & y_2''  \\
1 & x_3'' & y_3''  \\
\end{matrix} \right| b_a^2.
\end{gather}
Due to the relations in Lemma \ref{lm:-1three}, this equals
\begin{gather}
\frac{
\left| \begin{matrix}
1 &  x_1'' & {x_1''}^2 \\
1 &  x_2'' & {x_2''}^2 \\
1 &  x_3'' & {x_3''}^2 \\
\end{matrix} \right| }{
\left| \begin{matrix}
1 &  x_1' & {x_1'}^2 \\
1 &  x_2' & {x_2'}^2 \\
1 &  b_a & {b_a}^2 \\
\end{matrix} \right| }
\times 
\left| \begin{matrix}
1 & x_1' & y_1' & {x_1'}^2 \\
1 & x_2' & y_2' & {x_2'}^2 \\
1 & x_3' & y_3' & {x_3'}^2 \\
1 & b_a & 0   & b_a^2 \\
\end{matrix} \right|
=
\frac{
\left| \begin{matrix}
1 &  x_1'' & {x_1''}^2 \\
1 &  x_2'' & {x_2''}^2 \\
1 &  x_3'' & {x_3''}^2 \\
\end{matrix} \right| }{
\left| \begin{matrix}
1 &  x_1' & {x_1'}^2 \\
1 &  x_2' & {x_2'}^2 \\
1 &  b_a & {b_a}^2 \\
\end{matrix} \right| }
\times 
\left(
\left| \begin{matrix}
1 & x_1' & y_1' & {x_1'}^2 \\
1 & x_2' & y_2' & {x_2'}^2 \\
1 & b_a & C_a t_a & b_a^2 \\
1 & b_a & 0   & b_a^2 \\
\end{matrix} \right| + d_{\ge}(t_a^2) \right).
\end{gather}
We  consider the factors in this formula. 
The fact that $Q_1$ and $Q_2$ are generic for generic
$P_1$ and $P_2$, and (\ref{eq:III})
imply that both
$$
\left| \begin{matrix}
1 &  x_1'' & {x_1''}^2 \\
1 &  x_2'' & {x_2''}^2 \\
1 &  x_3'' & {x_3''}^2 \\
\end{matrix} \right|  \quad\mbox{and}\quad
\left| \begin{matrix}
1 &  x_1' & {x_1'}^2 \\
1 &  x_2' & {x_2'}^2 \\
1 &  b_a & {b_a}^2 \\
\end{matrix} \right| 
$$
do not vanish in the limit $P_3 \to B_a$.
Hence  $A_a$ has a simple zero at $B_a$.

To prove (4), we consider $P_1$ and $P_2$ near $B_a$;
$P_i = (x_i, y_i)$ behaves like 
$(t_b^3 + b_a, \zeta_3^i C_a t_b + d_{\ge}(t_b^2) )$,
$i=1, 2$.
Let us consider the behavior of $A_a$ and $F_a$.
\begin{gather*}
\begin{split}
A_a( P_1, P_2, P_3)
&= \displaystyle{
\frac{
\left| \begin{matrix}
1 & b_a + t_a^3 + d_{\ge}(t_a^4)& \zeta_3 C_a t_a + d_{\ge}(t_a^3) 
& (b_a + t_a^3)^2+d_{\ge}(t_a^4) \\
1 & b_a + t_a^3 + d_{\ge}(t_a^4)& \zeta_3^2 C_a t_a + d_{\ge}(t_a^3) 
& (b_a + t_a^3)^2+d_{\ge}(t_a^4) \\
1 & x_3 & y_3 & x_3^2 \\
1 & b_a & 0   & b_a^2 \\
\end{matrix} \right|}
{
\left| \begin{matrix}
1 & b_a + t_a^3 + d_{\ge}(t_a^4)& \zeta_3 C_a t_a + d_{\ge}(t_a^2) \\
1 & b_a + t_a^3 + d_{\ge}(t_a^4)& \zeta_3^2 C_a t_a + d_{\ge}(t_a^2) \\
1 & x_3 & y_3  \\
\end{matrix} \right|
}}\\
\end{split}
\end{gather*}
\begin{gather*}
\begin{split}
&= \displaystyle{
\frac{
\left| \begin{matrix}
0 & 0   & (\zeta_3-\zeta_3^2)C_a + d_{\ge}(t_a^3) & 0   \\
1 & b_a + t_a^3 + d_{\ge}(t_a^4)& \zeta_3^2 C_a t_a + d_{\ge}(t_a^3) 
& (b_a + t_a^3)^2+d_{\ge}(t_a^4) \\
1 & x_3 & y_3 & x_3^2 \\
1 & b_a & 0   & b_a^2 \\
\end{matrix} \right|}
{
\left| \begin{matrix}
1 & b_a + t_a^3 + d_{\ge}(t_a^4)& \zeta_3 C_a t_a + d_{\ge}(t_a^2) \\
0 & 0                     & (\zeta_3^2-\zeta_3) C_a t_a + d_{\ge}(t_a^2) \\
1 & x_3 & y_3  \\
\end{matrix} \right|
}}.\\
\end{split}
\end{gather*}

Direct computation shows that this equals $t_a^3 (b_a - x_3)$,
and in turn $F_a(P_1, P_2, P_3)=t_a^6 (b_a - x_3)$.
\end{proof}

We note that 
 (1) and (3) in Proposition \ref{prop:Azero}
give the multiplicity of zeros and poles of
$F_a$ when
 viewed as a function over $X \times S^2 X$.
In particular, $F_a$ does not vanish at 
$[-1]_3([-1]_2(P_1, P_2), B_a)$.

\bigskip

\section{The sigma function}
We  introduce the $\sigma$ function corresponding to $X$,
 an entire function over $\kappa^{-1}(\JJ)$,
following \cite[Section 3]{EEMOP1, EEMOP2}.
We recall the definition of $\sigma$ and its properties without
proofs.

We introduce the period matrices by
\begin{equation}
   \left[\,\eta'  \ \eta''  \right]=\frac{1}{2} 
\left[\int_{\alpha_i}\nuII_j \ \ \int_{\beta_i}\nuII_j
\right]_{i,j=1,2,3},
   \label{eq2.5}
\end{equation} 
where $\nuII_j$'s are the differentials of the
second kind \cite[(1.21) and (1.22)]{EEMOP2},
\begin{equation*}
 \nuII_1(x,y)=\frac{x^2}{3 y^2}dx,\quad
 \nuII_1(x,y)=\frac{2xy}{3 y^2}dx,\quad
 \nuII_1(x,y)=\frac{(5 x^2 + 3 \lambda_3 x + \lambda_2)y}{3 y^2}dx.
\end{equation*}

\begin{proposition}
The matrix,
\begin{equation}
   M := \left[\begin{array}{cc}2\omega' & 2\omega'' \\ 
                              2\eta' & 2\eta'' \end{array}\right],
\end{equation} 
 satisfies 
\begin{equation}
   M\left[\begin{array}{cc} & -1 \\ 1 & \end{array}\right]{}^t {M}
   =2\pi\sqrt{-1}\left[\begin{array}{cc} & -1 \\ 1 &
     \end{array}\right].
   \label{eq2.7}
\end{equation} 
\end{proposition}
This provides a symplectic structure in the Jacobian
which is known as  {\it generalized Legendre relation}
\cite{Ba1, BEL2}.
It is known that ${\omega'}^{-1}\omega''$ is a
symmetric, positive-definite matrix. 

As shown by  Riemann \cite{F1}, 
 $\text{Im}\,({\omega'}^{-1}\omega'') $ is positive definite.
Noting Theorem 1.1 in \cite{F1}, let
\begin{equation}
   \delta:=\left[\begin{array}{cc}\delta''\ \\
       \delta'\end{array}\right]\in \left(\tfrac12\ZZ\right)^{6}
   \label{eq2.9} 
\end{equation} 
be the theta characteristic which gives the Riemann constant with
respect to the base point $\infty$ and the period matrix 
$[\,\omega'\ \omega'']$.

For $u \in \mathbb C^g$, we define
\begin{eqnarray}
   \sigma(u)&=&\sigma(u;M)=\sigma(u_1, u_2,u_3;M) 
   \label{def_sigma}\\
   &=&c\,\text{exp}(-\tfrac{1}{2}u\eta'{\omega'}^{-1}\ ^t\negthinspace u)
   \vartheta\negthinspace
   \left[\delta\right](\frac{1}{2}{\omega'}^{-1}\ ^t\negthinspace u;\ 
{\omega'}^{-1}\omega'') \nonumber \\
   &=&c\,\text{exp}(-\tfrac{1}{2}u\eta'{\omega'}^{-1}\ ^t\negthinspace u)
\nonumber  \\
   &\times&
   \sum_{n \in \ZZ^3} \exp \big[\pi \ii\big\{
    \ ^t\negthinspace (n+\delta''){\omega'}^{-1}\omega''(n+\delta'')
   + \ ^t\negthinspace (n+\delta'')({\omega'}^{-1}\,u+\delta')\big\}\big],
\nonumber  
\end{eqnarray}
where  $c$ is a certain constant.
In this article, the constant $c$ is chosen in such a way that
the local expansion of $\sigma$
is consistent with  Proposition \ref{prop:sigma}  (2).

For a given $u\in\CC^3$, we introduce the notation
$u'$ and $u''$ for the  $\RR^3$-vectors such that
\begin{equation*}
   u=2\omega'u'+2\omega''u''.
\end{equation*}

A `shifted theta divisor' $\Theta_{2}$ is the vanishing 
locus of $\sigma$;
\begin{equation}
	\Theta_{2} = \WW_{2} \cup [-1] \WW_{2} =\WW_{2}.
\label{eq:Theta:g-1}
\end{equation}

Here we summarize the  properties of $\sigma(u;M)$ as follows: 
\begin{proposition} 
\label{prop:sigma} 
For all $u\in\CC^3$, $\ell\in\Lambda$,
and $\gamma\in\mathrm{Sp}(6,{\ZZ})$, we have\,{\rm :} 
\begin{enumerate}
\item
\begin{equation*}
    \sigma(u;\gamma M)=\sigma(u;M).
\end{equation*}

\item
      $u\mapsto\sigma(u;M)$  has zeroes of
order $1$ along $\Theta_{2} = \WW_2$;
$$
\sigma(u;M)=0 \iff u\in\Theta_{2}.
$$
$\sigma$ has the following expansion near $\Theta_{2}$,
\begin{equation*}
	\sigma(u) = (u_1 - u_3 u_2^2 + \frac{1}{20} u_3^5) + 
        \mbox{higher-weight terms}.
\label{eq:sigmaext}
\end{equation*}

\item
For $u$, $v\in\CC^3$, and $\ell$
($=2\omega'\ell'+2\omega''\ell''$) $\in\Lambda$, we define
\begin{align*}
  L(u,v)    &:=2{}^t{u}(\eta'v'+\eta''v''),\nonumber \\
  \chi(\ell)&:=\exp[\pi\sqrt{-1}\big(2({}^t {\ell'}\delta''-{}^t
  {\ell''}\delta') +{}^t {\ell'}\ell''\big)] \ (\in \{1,\,-1\}).
\end{align*}
The following holds
\begin{equation*}
	\sigma(u + \ell) = \sigma(u)
               \exp(L(u + \frac{1}{2}\ell, \ell)) \chi(\ell).
        \label{eq:4.11}
\end{equation*}

\item
For the  action $\hzeta_3$ on $u$ 
defined in Section \ref{curves},
we have
\begin{equation*}
\sigma_{}(\hzeta_3 u)= \zeta_3 \sigma_{}(u),
\quad
\hzeta_3 \Theta_2= \Theta_2.
\label{eq:zeta}
\end{equation*}

\end{enumerate}
\end{proposition}

\begin{proof}
See \cite{EEMOP2}.
\end{proof}

\begin{proposition}
\label{prop:pperiod} 
For $\ell\in \Lambda$ as in Proposition \ref{prop:sigma}, 
\begin{equation}
	\sigma(u + \hzeta_3 \ell) = \sigma(u)
               \exp(L(\hzeta_3^2 u + \frac{1}{2}\ell, \ell)) \chi(\ell).
        \label{eq:4.11zeta}
\end{equation}
\end{proposition}

\begin{proof}
By considering $\sigma(\hzeta_3 v + \hzeta_3 \ell) = 
 \zeta_3\sigma(v + \ell)$ and letting $v = \hzeta_3^2 u$,
 we have 
$$
\sigma(\hzeta_3 v + \hzeta_3 \ell) = \zeta_3 \sigma(v)
               \exp(L(v + \frac{1}{2}\ell, \ell)) \chi(\ell).
$$
\end{proof}

\section{Addition law. II}
We recall the following results from the Appendix of \cite{EEMOP1}.


Following \cite{O}, we introduce
the partial derivative over a multi-index $\natural^n$,
$$
	\sigma_{\natural^n}(u)= 
\left\{\begin{matrix} 
\frac{\partial^2}{\partial u_3^2} \sigma(u)=\sigma_{33}(u)
 & \mbox{for } n = 1, \\
\frac{\partial}{\partial u_3} \sigma(u)=\sigma_{3}(u) 
& \mbox{for } n = 2, \\
 \sigma(u) & \mbox{for } n > 2. \\
\end{matrix}\right.
$$
Further we note that for $u\in \kappa^{-1}\WW_n$ 
and the action $\hzeta_3$ on $u$ equivariant
under the Abel map with the action
 $(x_i, y_i) \mapsto (x_i, \zeta_3 y_i)$, as in Proposition \ref{prop:sigma},
part (\ref{eq:zeta}),
 we have
\begin{equation}
\sigma_{}(\hzeta_3 u)= \zeta_3 \sigma_{}(u),
\quad
\sigma_{\natural^2}(\hzeta_3 u)=  \sigma_{\natural^2}(u),
\quad
\sigma_{\natural^1}(\hzeta_3 u)= \zeta_3^2 \sigma_{\natural^1}(u),
\label{eq:sign2}
\end{equation}
due to \cite[(A.2)]{EEMOP1}.

\begin{definition} 
For a positive integer $n>1$ and
a point $(x_1,y_1), \ldots, (x_n, y_n)$ in $X^n$,
we define
\begin{gather*}
\begin{split}
&\Delta_n((x_1, y_1), \ldots, (x_n, y_n))\\
& := 
\left|
\begin{matrix}
1 & \phi_1(x_{1}, y_{1}) & \phi_2(x_{1}, y_{1}) &
\cdots & \phi_{n-1}(x_{1}, y_{1})  \\
1 & \phi_1(x_{2}, y_{2}) & \phi_2(x_{2}, y_{2}) &
\cdots & \phi_{n-1}(x_{2}, y_{2})  \\
\vdots & \vdots & \vdots& \ddots & \vdots  \\ 
1 & \phi_1(x_{n}, y_{n}) & \phi_2(x_{n}, y_{n}) &
\cdots & \phi_{n-1}(x_{n}, y_{n})  \\
\end{matrix}
\right|
\left|
\begin{matrix}
1 & x_{1} & x_{1}^2 & \cdots & x_{1}^{n-1}  \\
1 & x_{2} & x_{2}^2 & \cdots & x_{2}^{n-1}  \\
\vdots & \vdots &\vdots&  \ddots & \vdots \\ 
1 & x_{n} & x_{n}^2 & \cdots & x_{n}^{n-1}  \\
\end{matrix}
\right|. \\
\end{split}
\end{gather*}
\end{definition} 
Then we have
\begin{gather*}
\Delta_4((x_1, y_1), (x_2, y_2), (x_3, y_3), (x, y))
 = 
\left|
\begin{matrix}
1 & x_{1} & y_{1} & x_{1}^2 \\
1 & x_{2} & y_{2} & x_{2}^2 \\
1 & x_{3} & y_{3} & x_{3}^2 \\
1 & x_{4} & y_{4} & x_{4}^2 \\
\end{matrix}
\right|
\cdot\left|
\begin{matrix}
1 & x_{1} & x_{1}^2 & x_{1}^{3}  \\
1 & x_{2} & x_{2}^2 & x_{2}^{3}  \\
1 & x_{3} & x_{3}^2 & x_{3}^{3}  \\
1 & x_{ } & x_{ }^2 & x_{ }^{3}  \\
\end{matrix}
\right|,
\end{gather*}
\begin{gather*}
\Delta_3((x_1, y_1), (x_2, y_2), (x_3, y_3))
 = 
\left|
\begin{matrix}
1 & x_{1} & y_{1} \\
1 & x_{2} & y_{2} \\
1 & x_{3} & y_{3} \\
\end{matrix}
\right|
\cdot\left|
\begin{matrix}
1 & x_{1} & x_{1}^2 \\
1 & x_{2} & x_{2}^2 \\
1 & x_{3} & x_{3}^2 \\
\end{matrix}
\right|, 
\end{gather*}
\begin{gather*}
\Delta_2((x_1, y_1),(x_2, y_2))
 = 
\left|
\begin{matrix}
1 & x_{1} \\
1 & x_{2}  \\
\end{matrix}
\right|^2 , \quad
\Delta_1((x_1, y_1)) = 1.
\end{gather*}

\begin{theorem}\label{th:addition}\cite[Theorem A.1]{EEMOP2}
Assume that $(m, n)$ is a pair of positive integers ($n, m > 1$).
Let $((x_i,y_i)_{i=1, \ldots, n},$ 
 $(x'_i,y'_i)_{i=1, \ldots, m})$
a point in $\Sym^n(X)\times \Sym^m(X)$
and its image under the Abel map be 
$(u, v) \in $ $\kappa^{-1} \WW_n\times \kappa^{-1} \WW_m$.
Then the following relation holds
\begin{gather}
\begin{split}
&\frac{\sigma_{\natural^{n+m}}(u + v) \sigma_{\natural^{n+m}}(u +\hzeta_3 v) 
 \sigma_{\natural^{n+m}}(u +\hzeta_3^2 v) }
{\sigma_{\natural^m}(u)^3 \sigma_{\natural^n}(v)^3} \\
&=
\frac{\prod_{i=0}^2 \Delta_{m+n}
((x_1,y_1),\ldots, (x_m, y_m), 
(x_1',\zeta_3^i y_1'),\ldots,(x_n',\zeta_3^i y_n'))}
{(\Delta_{n}((x_1',y_1')\ldots, (x_n', y_n'))
\Delta_{m}((x_1,y_1)\ldots, (x_m, y_m)))^3}\\
&\times
\prod_{i=1}^m\prod_{j=1}^n 
\frac{1}{ \Delta_2((x_i, y_i), (x_j', y'_j))^2}.
\label{eq:Th3}
\end{split}
\end{gather}
\end{theorem}

\section{The trigonal $\mathrm{al}$ function for a cyclic trigonal curve}

We define the $\al$ function
for a cyclic trigonal curve $X$ and derive some
properties.


\begin{definition}
We introduce  triple coverings $\JJ^{(a; c)}$of the Jacobian $\JJ$,
$$
 \JJ^{(a; c)} := \CC^3 / \Lambda^{(a; c)}, \quad a = 1, 2, 3, 4,
\quad c=0, 1, 2,
$$
where 
\begin{gather*}
  \Lambda^{(a; c)} := 
        \sum_{b=1,2,3} (k^{(c)\prime}_{a,b} \ZZ \omega_b' 
        + k^{(c)\prime\prime}_{a,b} \ZZ \omega_b''). 
\quad
\end{gather*}
For brevity, strokes as 
$(\prime, \prime\prime)$  or $(\prime\prime, \prime)$
are denoted by $(\gamma, \bar\gamma)$,
and for $h$'s in (\ref{eq:invL}),
\begin{gather*}
        k^{(c)\gamma}_{a,b}  = 
\left\{
\begin{matrix}
 2 & \mbox{if } h^{(c)\bar\gamma}_{a,b}  = 0, \\
 6 & \mbox{if } h^{(c)\bar\gamma}_{a,b}  \neq 0.
\end{matrix}
\right.
\end{gather*}

These triple coverings correspond to
$\hzeta_3^c\omega_{b}$ and give two sets of natural projections:
$$
\varpi_{b; c} : \JJ^{(b; c)} \to \JJ, \quad
\kappa_{b;c} : \CC^3 \to \JJ^{(b; c)}.
$$
\end{definition}

We discuss the preimage of the Abel map 
under the projections $\varpi_{b; c}$ of $\JJ^{(b; c)}$ 
$(b = 1, 2, 3, c = 0, 1, 2)$ 
 in Section \ref{sec:Domain}.
If we further define 
$$
 \JJ^{\sharp}  = \CC^3 / \Lambda^{\sharp}, 
\quad
  \Lambda^{\sharp}:= \bigcap_{a = 1, 2, 3, 4; c = 0, 1, 2}
   \Lambda^{(a;c)},
\quad
\kappa_{\sharp} : \CC^3 \to \JJ^{\sharp},
$$
then $\JJ^{\sharp}$ is the smallest torus that covers
 each $\JJ^{(b;c)}$ $(b = 1, 2, 3, c = 0, 1, 2)$.
We could adapt the following theorem to $\JJ^{\sharp}$, cf.
 \cite[Ch.~III.7]{Mum}.
Similarly, $\JJ^{(0)}$, 
$$
 \JJ^{(0)}  = \CC^3 / \Lambda^{(0)},\quad 
  \Lambda^{(0)} :=  \sum_{i=1,2,3} (6\ZZ \omega_i' + 6 \ZZ \omega_i'' ),
$$
is a $3^{6}$-order covering of
$\JJ$; we sometimes consider meromorphic functions on this torus. 



\begin{definition}
For $(b = 1, 2, 3, 4, c = 0, 1, 2)$,
we define a meromorphic function on $\CC^3$
\begin{gather}
\begin{split}
\al_a^{(c)}(u)&:= 
A_{b, c}\frac{\ee^{-{}^t u \varphi_{a;c}}
       \sigma(u + \hzeta_3^c\omega_b)}{\sigma(u)}\\
 &=\frac{\ee^{- {}^t u \varphi_{a;c}}
       \sigma(u + \hzeta_3^c\omega_a)}
{\sigma(u) \sigma_{33}(\hzeta_3^c\omega_a)},\\
\label{eq:alb}
\end{split}
\end{gather}
where $A_{b, c}:=1/ \sigma_{33}(\hzeta_3^c\omega_b)$ and
\begin{gather*}
\varphi_{a;c} := \frac{2}{3}
\sum_{b=1}^3 (h^{(c)\prime}_{a, b} \eta' \omega^{\prime -1} \omega_b' + 
      h^{(c)\prime\prime}_{a, b} \eta'' \omega^{\prime\prime -1} \omega_b'')
\in \frac{1}{3}\Lambda.
\end{gather*}
\end{definition}


The
following Proposition 
shows that  the domains of these functions are chosen naturally,
as follows from
properties of $\sigma$;
Propositions \ref{prop:sigma} and \ref{prop:pperiod}
yield the periodicity  of the $\al$-functions:

\begin{proposition} \label{prop:al_period}
\begin{enumerate}
\item For a lattice point $\ell$ in 
$\Lambda^{(a;c)}$,
we have
$$
\al_a^{(c)}(u)=\al_a^{(c)}(u+ \ell).
$$
\item  $\al_a^{(c)}(u)$ is a function over the covers $\JJ^{(a;c)}$
of the Jacobian,  thus  a fortiori on $\JJ^{\sharp}$.

\item For $u \in \JJ^{\sharp}$,
$$
\al_a^{(c)}(u)=
\ee^{-\hzeta_3^c {}^t u \varphi_{a;c} +\hzeta_3^{-c} {}^t u \varphi_{a;0}}
\al_a^{(0)}(\hzeta_3^{-c} u).
$$
\end{enumerate}
\end{proposition}

\begin{proof}
 First we note that
$$
 \frac{\sigma(u + \hzeta_3^c\omega_a+ \ell)}{\sigma(u+\ell)} = 
\frac{\sigma(u + \hzeta_3^c\omega_a)}{\sigma(u)}
 \exp(L(\hzeta_3^c\omega_a , \ell)) .
$$
Due to (\ref{eq:invL}) and Proposition \ref{prop:sigma} (3),
$L(\hzeta_3^{c} \omega_a , \ell) $ is given by 
\begin{gather}
\begin{split}
 & \frac{2}{3}\left(
\sum_{b=1}^3 h^{(c)\prime}_{a,b}2^t\omega_b' (\eta'\ell' +\eta''\ell'') +
\sum_{b=1}^3 h^{(c)\prime\prime}_{a,b}2^t\omega_b'' (\eta'\ell' +\eta''\ell'')
\right) \\
 =& \frac{2}{3}
\sum_{b=1}^3 
\left( h^{(c)\prime}_{a,b}
(2^t\omega_b' \eta'\ell' +2^t\omega_b''\eta'\ell''+\pi\ii\ell''_b)
+ h^{(c)\prime\prime}_{a,b}
(2^t\omega_b' \eta'\ell'' + 2^t\omega_b''\eta''\ell''-\pi\ii\ell'_b)\right),\\
\end{split}
\end{gather}
whereas noting that 
${}^t\omega^\alpha \eta^\beta$ is unchanged under a switch 
${}^t\omega^\alpha \eta^\beta = {}^t\eta^\beta \omega^\alpha$,
where $\alpha, \beta  \in \{${\lq\lq}$\prime${\rq\rq}, 
{\lq\lq}$\prime\prime${\rq\rq}$\}$,
we have
\begin{gather}
\begin{split}
 {}^t\ell \varphi_{b;c} &=
 \frac{2}{3}
 2^t(\omega'\ell' +\omega''\ell'')
\sum_{b=1}^3 
\left(h^{(c)\prime}_{a, b} \eta' \omega^{\prime -1} \omega_b' + 
      h^{(c)\prime\prime}_{a, b} \eta'' \omega^{\prime\prime -1} \omega_b''
\right)\\
 &= \frac{2}{3}
\sum_{b=1}^3\left(
 h^{(c)\prime}_{a, b}
 2^t\ell'{}^t\eta'\omega_b'+2{}^t\ell''{}^t\eta'\omega_b'')
 +h^{(c)\prime\prime}_{a, b}
 (2^t\ell'{}^t\eta''\omega_b'+2{}^t\ell''{}^t\eta''\omega_b'')\right).\\
\end{split}
\end{gather}
The difference between 
$L(\hzeta_3^{c} \omega_a , \ell) $ and ${}^t\ell \varphi_{b;c}$
vanishes modulo $2\pi\ii\Lambda^{(a;c)}$.
Hence (2) is obvious;
(3) is straightforward.
\end{proof}



\begin{remark}{\rm{
We have  a more general  $\al$ function defined  for 
the cover $\JJ^{(0)}$  of the Jacobian 
\begin{gather}
\al\left[\begin{matrix} c_1' & c_1''\\ c_2' & c_2''\\ 
            c_3' & c_3''\end{matrix}\right] (u):=
 \frac{\ee^{-\hzeta_3^c {}^t u \varphi(c_a'; c_a'')}
   \sigma(u + \sum_{a}c_a'\omega_{a}'+\sum_{a}c_a''\omega_{a}'' )
    }
{\sigma(u) } ,
\label{eq:albs}
\end{gather}
where $c_a'$ and $c_a''$ are $0$, $1$, $\zeta_3$, or $\zeta_3^2$,
and $\varphi(c_a'; c_a'')$ is an appropriate vector of $\CC^g$.
The shift shows that the functions
are associated to theta functions with characteristics.
}}
\end{remark}


\begin{lemma}
\label{lemma:zero_al}
\begin{enumerate}
\item
$
 \deg_{\kappa^{-1}( \hzeta_3^{c'}\omega_a + \Theta_2)} \al_a^{(c)}(u) = 0
$ for $c'=0,1,2$,

\item
$
 \deg_{\kappa^{-1}(- \hzeta_3^{c'}\omega_a+ \Theta_2)} \al_a^{(c)}(u) = 1
$ for $c'=0,1,2$,

\item
$
 \deg_{\kappa^{-1}(\Theta_2)} \al_a^{(c)} = -1
$, and

\item 
$ \kappa^{-1}(- \hzeta_3^{c'}\omega_a+ \Theta_2)=
\kappa^{-1}(- \omega_a+ \Theta_2)$ for $c'=0,1,2$.

\end{enumerate}
\end{lemma}
\begin{proof}
The zero divisor of $\sigma$  is $\kappa^{-1}\Theta_2$ and thus for
$(\Gamma_{P_1,\infty}, \Gamma_{P_2,\infty}, \Gamma_{P_3,\infty})$ in
$S^3 \Gamma_\infty X$,
$w^{-1} \kappa^{-1}\Theta_2$ corresponds to points
$\{\Gamma_{\infty,\infty}, \Gamma_{P_2,\infty}, \Gamma_{P_3,\infty}\}$
if fixing 
$(\Gamma_{P_2,\infty}, \Gamma_{P_3,\infty})$.
Hence  $\al$ as a function of $\Gamma_{P_1,\infty}$ has only 
a simple zero at
one point  in each lattice $\Lambda$.
(1) and (3) are obvious.

(2) follows immediately from (4) if $c=c'$.
More generally, we show (2) using (4) as follows.
Let us consider the case of $a=1$: 
For $c=2$ and $c'=0$, $(-1+\hzeta_3^2) \omega_1=2(\omega'_1+\omega_1''
-\omega_3'') \in \Lambda$.
For $c=1$ and $c'=0$, $(-1+\hzeta_3) \omega_1=$
$(1-\hzeta^2)(-1+\hzeta_3^2) \omega_1 \in \Lambda$. Hence 
$\kappa^{-1}(- \hzeta_3^{c'}w(B_a)+ \omega_a +\Theta_2)$
$= \kappa^{-1}(\Theta_2)$.
Similarly we have the other cases $a=2, 3$, cf. Figure 2.
We can see geometrically that  (4) holds  because
if $ \kappa^{-1}(- \hzeta_3^{c'}\omega_a + \Theta_2) \neq
\kappa^{-1}(- \omega_a + \Theta_2)$ for $c'=1,2$, there are two points
in a fundamental domain of
$\Lambda$ which are zeros of the numerator: this
contradicts the properties of the sigma function in 
Proposition \ref{prop:sigma} (2) and (3).
\end{proof}

For a point $\Gamma_{P_c, \infty}$ in $\Gamma_\infty X$ and
a local parameter
$t_a$, $(t_a^{(c)})^3 = (x_c - b_a)$ for $P_c=(x_c,y_c)$ in $X$,
$t_a$ is transformed 
to $\zeta_3t_a^{(c)}$ 
for a loop around $B_a$ in $\Gamma_\infty X$.
We let the function 
$$
\varepsilon_{a}^{(c)}: \Gamma_{\infty} X \to \ZZ_3
$$
 be defined by $\varepsilon_a^{(c)}:=w_a-w_\infty$ modulo $3$
for the winding number $w_a$
around $B_a$ in $\kappa_\infty\Gamma_\infty X$ and
for the winding number $w_\infty$
 around $\infty$ in $\kappa_\infty \Gamma_\infty X$.
Using it, we also define
$$
\varepsilon_a: S^3 \Gamma_{\infty} X \to \ZZ_3, \ \
(\varepsilon_a:= \varepsilon_a^{(1)} 
+ \varepsilon_a^{(2)} + \varepsilon_a^{(3)}) \mbox{ over }
w^{-1} \kappa^{-1}(-w(B_a)+ \Theta_2) \mbox{ and }
w^{-1} \kappa^{-1}(\Theta_2).
$$


Our first main theorem  is:
\begin{theorem}
\label{th:al}
For a point
$(\Gamma_{P_1, \infty}, \Gamma_{P_2, \infty}, \Gamma_{P_3, \infty})$ 
in $w^{-1}(\JJ^{(a;c)})$ as a subset of a
quotient space of $S^3 \Gamma_\infty X$,
\begin{gather}
\al_a^{(c)}(u)
= -\zeta_3^{c+\varepsilon_a(
\Gamma_{P_1, \infty}, \Gamma_{P_2, \infty}, \Gamma_{P_3,\infty})}
\frac{A_a(P_1, P_2, P_3)}{\sqrt[3]{F_a(P_1, P_2, P_3)}} ,
\label{eq:ala}
\end{gather}
with a first-order pole at  $\varpi_{a,c}^{-1}\Theta_2$
and a simple zero  at $-\omega_a+\varpi_{a,c}^{-1}\Theta_2$.
\end{theorem}

\bigskip
\bigskip


\begin{remark} \label{rmk:8.7}
{\rm{
Before we prove the theorem, we comment on the cubic root and 
$\zeta_3^\varepsilon$
in the right-hand side of  (\ref{eq:ala}).
We need a choice of  cubic root, so the function is not
 defined over the algebraic space 
$S^3 X$.  However since $X$ is given by $y^3 = f(x)$,
we will see below (Lemma \ref{lm:alWelldef})
that  over $S^3 \Gamma_\infty X$ we can make a specific choice
and define a global function. We observe the following:

In view of Definition \ref{eq:defAF}, $A_a$ and $F_a$ are invariant 
under the action $\hzeta_3: S^3 X \to S^3 X$, {\it{i.e.}},
$\hzeta_3(P_1, P_2, P_3) =(\hzeta_3 P_1, \hzeta_3 P_2, \hzeta_3 P_3)$,
when $(P_1, P_2, P_3)$ is a generic point in $S^3 X$.
The action $\hzeta_3$ induces 
$\hzeta_3: \Gamma_\infty X \to \Gamma_\infty X$, so that it moves
a point $P \in \Gamma_\infty X$ to 
another point $P' \in \Gamma_\infty X$ satisfying 
$\kappa_\infty(P) =\kappa_\infty(P')$.

As  mentioned in Remark \ref{rmk:holo} (2), 
we have a $\ZZ_3$ action on each local parameter 
$t_a^{(c)}$, $(t_a^{(c)})^3 = (x_c - b_a)$ and
the action $\hat \iota^{(a)}_{\zeta_3} : t_a^{(c)} \to \zeta_3 t_a^{(c)}$
 is locally identified with  $\hat \zeta_3$. Therefore, we can define
the cubic root of $F_a$ over $S^3\Gamma_\infty X$.

Further, $1/t_a^{(c)}$ is a local parameter at $\infty$ and 
we define  $\iota^{(\infty)}_{\zeta_3} : 1/t_a^{(c)} \to \zeta_3 1/t_a^{(c)}$
as a local biholomorphic map.
A circuit around the point transforms the divisor into
$\zeta_3^{\varepsilon_a(
\Gamma_{P_1, \infty}, \Gamma_{P_2, \infty}, \Gamma_{P_3,\infty})}$.

We check the consistency of the factor $\zeta_3^c$
and the global definedness in Lemma \ref{lm:alWelldef};
here we informally interpret the right-hand side as follows:
Around $(B_a,P_2,P_3)$,
 $\sqrt[3]{F_a(P_1, P_2, P_3)}$ is given by
$t_a^{(1)} a_1(P_2,P_3)$ and
$A_a(P_1, P_2, P_3)$ as $t_a^{(1)} a_2(P_2,P_3) + \cdots$
where $a_i(P_2,P_3)$, $i=1,2$,
is a non-vanishing function of $P_2$ and $P_3$,
and thus the two factors cancel.
Around $[-1](B_a,[-1](P_2,P_3))$,
$\sqrt[3]{F_a(P_1, P_2, P_3)}$ does not vanish whereas
$A_a(P_1, P_2, P_3)$ behaves like $t_a^{(1)} 
a_3(P_2,P_3) + \cdots$ 
and thus the path around the point generates
$\zeta_3^{\varepsilon_a(
\Gamma_{P_1, \infty}, \Gamma_{P_2, \infty}, \Gamma_{P_3,\infty})}$,
 where $a_3(P_2,P_3)$ is a non-vanishing function of $P_2$ and $P_3$.
}}
\end{remark}

\begin{remark} \label{rmk:8.8}
{\rm{
In order to find the  $\sigma$ function 
on $\CC^3 = \kappa^{-1}\JJ$, we use the 
$\al_a^{(c)}$-function, which is also defined over a covering space of
$\JJ$; indeed, the 
$\al_a^{(c)}$-function involves a field extension of
  meromorphic functions on $\JJ$, 
using the Galois group action $\ZZ_3$, according to   
 Weierstrass' construction  in \cite{Wei}.

Mumford gave three types of meromorphic functions on 
a Jacobian variety, defined
by theta functions, cf. \cite[Ch.~II.3]{Mum}.
One type  (Method III in {\textit{loc. cit.}},
 a second logarithmic derivative of theta) 
is a generalization of 
the elliptic $\wp$ 
function. The corresponding function for a  hyperelliptic curve  was studied  
 in \cite[Ch.~III]{Mum} and \cite{P} in terms of theta functions.
the 19th century
\cite{Kl1,Ba1,Ba2,Ba3}.
This type  is related to KdV hierarchy and KP hierarchy. In fact Baker
 found the KdV hierarchy and KP equation,
 though not identifying their origin
 as non-linear wave equations, cf. \cite{Ba1,Ba3,BEL1,Ma0}.
 A second and third type (Method II and I resp. in {\textit{loc. cit.}},
the logarithmic derivative of a quotient of theta functions with
characteristics
and a quotient of products of theta functions translated
by linearly equivalent divisors, resp.) are related to the
sn, cn, dn functions in the elliptic curve case and Weierstrass' $\al$
function in the case of hyperelliptic curves.
Type II (Method II) is associated with the modified KdV equation \cite{Ma01}.
Type III (Method I) corresponds to the polynomial 
$U$-function of the triple called
$(U,V,W)$  in \cite[Ch.~III]{Mum}. The square root of $U$ is Weierstrass'
$\al$ function, which is associated with the sine-Gordon
equation and the Neumann system. Weierstrass discovered his 
version of the sigma function,
$Al$, in terms of his $\al$ function.

As mentioned in the Introduction, the trigonal
$\al$ function will provide properties of the abelian-function theory
 of the curve $X$.
In fact, we obtain an identity for the sigma function in 
Theorem \ref{th:Frob} below.

It should be noted that our method to investigate the sigma function
or theta function in terms of the $\al$-function can be generalized to
more general Galois curves.
}}
\end{remark}

\bigskip


\begin{proof}
We consider the case $n=3$ and $m=1$ and $\omega_a = v((b_a, 0))$
in Theorem \ref{th:addition}.
Then the left-hand side of (\ref{eq:Th3}) is equal to
\begin{gather*}
\frac{\sigma_{\natural^{n+m}}(u + v) \sigma_{\natural^{n+m}}(u +\hzeta_3 v) 
 \sigma_{\natural^{n+m}}(u +\hzeta_3^2 v) }
{\sigma_{\natural^m}(u)^3 \sigma_{\natural^n}(v)^3} 
=
\frac{\sigma(u + \omega_a)\sigma(u +\hzeta_3 \omega_a)
\sigma(u + \hzeta_3^2\omega_a) }
{\sigma(u)^3 \sigma_{\natural^1}(\omega_a)^3} 
\end{gather*}
whereas the right-hand side of (\ref{eq:Th3}) 
 is given as.
$$
A_a(P_1, P_2, P_3)^3 
F_a(P_1, P_2, P_3)^3 
\frac{1}{
(x_1 - b_a)^4
(x_2 - b_a)^4
(x_3 - b_a)^4}.
$$
We  note that 
$\hzeta_3 \omega_a \neq \omega_a$.

The zeros and poles are given by
Lemma \ref{lemma:zero_alrh}, consistent with 
  Lemma \ref{lemma:zero_al}.
Noting that $1 + \zeta_3 + \zeta_3^2 = 0$ and
$1 + \hzeta_3 + \hzeta_3^2 = 0$,
 the periodicity is determined.


The domains of  both  sides coincide due to
Remark \ref{rmk:8.7} and Lemma \ref{lm:alWelldef}.


The identity gives the following equality
up to a constant factor $K_{a,c}$,
\begin{equation}
\al_a^{(c)}(u)
=K_{a, c}
\zeta_3^{\varepsilon_a(
\Gamma_{P_1, \infty}, \Gamma_{P_2, \infty}, \Gamma_{P_3, \infty})}
\frac{A_a(P_1, P_2, P_3)}{\sqrt[3]{F_a(P_1, P_2, P_3)}} .
\label{eq:alaKa}
\end{equation}
Lemma \ref{lemma:Ka} 
and
Lemma \ref{lm:alWelldef}
give the factor $K_{a, c}$ and
$\zeta_3^{\varepsilon_a}$ respectively.

\end{proof}

\begin{lemma}
\label{lemma:zero_alrh} 
We obtain the following multiplicities
for the right-hand side of (\ref{eq:ala}):
\begin{enumerate}
\item
$
 \deg_{w^{-1}(\hzeta_3^{c'} \omega_a + \Theta_2)} 
\frac{A_a(P_1, P_2, P_3)}{\sqrt[3]{F_a(P_1, P_2, P_3)}} = 0,
$ for $c'=0,1,2$,
\item
$
 \deg_{w^{-1}(-\hzeta_3^{c'}\omega_a + \Theta_2)} 
\frac{A_a(P_1, P_2, P_3)}{\sqrt[3]{F_a(P_1, P_2, P_3)}} = 1,$
 for $c'=0,1,2$,
\item
$
 \deg_{w^{-1}(\Theta_2)} 
\frac{A_a(P_1, P_2, P_3)}{\sqrt[3]{F_a(P_1, P_2, P_3)}} = -1.  \quad
$
\end{enumerate}
\end{lemma}
\begin{proof}
Since the right-hand side of (\ref{eq:ala}) 
 and $\Theta_2$ are invariant for the  action of $\hzeta_3$,
the problem is reduced to Proposition \ref{prop:Azero},
 which gives the multiplicities  of zeros and poles of
$A_a$ and $F_a$.
\end{proof}


\begin{lemma}
\label{lm:alWelldef}
The domain of 
$\zeta_3^{\varepsilon_a(
\Gamma_{P_1, \infty}, \Gamma_{P_2, \infty}, \Gamma_{P_3, \infty})}$
$\displaystyle{\frac{A_a(P_1, P_2, P_3)}{\sqrt[3]{F_a(P_1, P_2, P_3)}}}$
is the preimages \\
$\{\Gamma_{P_1,\infty}, 
\Gamma_{P_2,\infty}, \Gamma_{P_3,\infty})\}$
under a `lifted' Abel map $w$ into
$\JJ^{(a; c)}\setminus \varpi_{a,c}^{-1}\Theta_2$, $(c = 0, 1, 2)$,
 thus a subset of a  quotient  of $S^3 \Gamma_\infty X$.
\end{lemma}

\begin{proof}
Let us consider the function of $\Gamma_{P_1,\infty} \in \Gamma_\infty X$ 
by fixing  
$\Gamma_{P_2,\infty}$ and $\Gamma_{P_3,\infty}$ of $\Gamma_\infty X$.
When we cross the cut out of  infinity, 
from (\ref{eq:AF3}), $A_a(P_1; P_2, P_3)$ acquires a $\zeta_3$ factor, 
which cancels the $\zeta_3$ factor of the denominator.
For the $a = 1, 2, 3$ case, Figure 1 shows that 
under a circuit along $\alpha_a$,    
the phase of $t_a^{(1)}$ does not change because: In the $a=1$ case,
the contour $\alpha_a$ does not have any effect on the phase factor of 
the $t_a^{(1)}$; In the
 $a=2, 3$ case,  passing through the crosscut adds the phase 
factor $\zeta_3$ of $t_a^{(1)}$ but  passing through  infinity 
compensates it. 
As for  the contour $\beta_a$, 
$t_a^{(1)}$ and $y_a$ are local parameters around 
$B_a = (b_a, 0)$ and  the $\zeta_3$-factors of 
numerator and denominator of $A_a/\sqrt[3]{F_a}$
cancel.
Similarly the circuit along  $\alpha_b$ and $\beta_b$ $(b\neq a)$
does not have any effect on the phase.
Due to Proposition \ref{lm:omega_a}, the $a=4$ case
is also checked.

However we see from Proposition \ref{prop:Azero} (2) that
around $\kappa_{\infty}^{-1}\{[-1]_3([-1]_2(P_1, P_2), B_a)\}$,
$A_a$ generates the factor $\zeta_3$ whereas $\sqrt[3]{F_a}$ does not
because $F_a$ does not vanish there. For
the factors to cancel we need to circle the branch point
$B_a$ three times. Hence 
the domain of 
$\displaystyle{
\zeta_3^{\varepsilon_a(
\Gamma_{P_1, \infty}, \Gamma_{P_2, \infty}, \Gamma_{P_3, \infty})}
}$ $\displaystyle{
\frac{A_a(P_1, P_2, P_3)}{\sqrt[3]{F_a(P_1, P_2, P_3)}}}$
can be viewed as a point of
the
triple symmetric product of
 $\Gamma_\infty X$ on which $\mathbb{Z}_3$ acts, as well as a quotient space.
The domain is the same as the preimage, under the extended  Abel
map, of $\JJ^{(a; c)}\setminus \varpi_{a,c}^{-1}\Theta_2$ in $\Gamma_\infty X$.
\end{proof}


We determine the factor $K_{a, c}$ in (\ref{eq:alaKa})
in the following Lemma. 
\begin{lemma}
\label{lemma:Ka}
$$
K_{a, c} =-\zeta_3^c, \quad
\sigma_{33}(\omega_a) =\frac{\sqrt{2}}{
\sqrt[3]{(\frac{d f(x)}{ dx})}\Bigr|_{x=b_a}}. 
$$
\end{lemma}

\begin{proof}
We use the notation $\frac{d f}{dx}|_{x=b_a} = C_a^3$.
In formula (\ref{eq:ala}), we let
 $u \mapsto  - \hzeta_3^c  \omega_a$ by addition
 $ -\hzeta_3^c  \omega_a + u^{(3)}$, where $u^{(3)} := w((x_3, y_3))$.
Using Proposition \ref{prop:Azero} (4),
the first-order approximation corresponds to
 $(P_1, P_2, P_3) \to
(\zeta_3^{c+1} (b_a, 0), \zeta_3^{c+2} (b_a, 0), (x_3, y_3))$.
Remark \ref{rmk:holo} shows that for the differentials,
$$
	\frac{\partial}{\partial u_1}
        = C_a^2
	\frac{\partial}{\partial t_b}, \quad
	\frac{\partial}{\partial u_3^{(3)}}
        = -\frac{\partial}{\partial t_\infty}.
$$
Using the notation in the proof of Proposition \ref{prop:Azero}, we have
 the following:
$$
	\frac{\partial}{\partial u_1}
 \ee^{-{}^t u\varphi_{b;c}}
        \frac{\sigma(u+\hzeta_3^c\omega_a)}
       {\sigma(u)\sigma_{33}(\hzeta_3^c \omega_a) }
 \Bigr|_{u=u^{(3)} -\hzeta_3^c\omega_a} =
 K_{a,c} C_a^2\sqrt[3]{(b_a - x_3)^2},
$$
{\it{i.e.}},
\begin{gather}
 \ee^{- {}^t u^{(3)} \varphi_{b;c}}
\frac{\sigma_{1}(u^{(3)})  }
    {\sigma(u^{(3)} - \hzeta_3^c\omega_a)\sigma_{33}(\hzeta_3^c \omega_a) }
   = K_{a, c} C_a^2 \sqrt[3]{(x_3 - b_a)^2}.
\label{eq:K_aproof1}
\end{gather}
For the computation of the left-hand side, 
we have used that  sigma vanishes on $\Theta_{2}$ but
$\sigma_{1}(u^{(3)})$ does not  vanish identically on $\WW_1$.

Similarly we differentiate the inverse of
(\ref{eq:K_aproof1}) in $t_\infty$ twice 
with respect to $u_3^{(3)}$,
$$
\frac{d^2}{du_3^{(3) 2}}
\left[
 \ee^{ {}^t u^{(3)} \varphi_{a;c}}
\frac{\sigma(u^{(3)} - \hzeta_3^c\omega_a)\sigma_{33}(\hzeta_3^c \omega_a) }
{\sigma_{1}(u^{(3)}) } \right]
=
\frac{d^2}{du_3^{(3) 2}}
\left[ \frac{1}{ K_{a, c} C_a^2} t_\infty^2  + d_{>}(t_\infty^3) \right].
$$
Here we note that $\sigma_\natural(u) =\sigma_{33}(u)$ and
from (\ref{eq:sigmaext}), 
$$
        \sigma_1(0) = 1, \quad
        \sigma_{33}(-\hzeta_3^c\omega_a)=- \zeta_3^{2c}\sigma_{33}(\omega_a). 
$$
Then we have
$$
-\zeta_3^{4c} \frac{\sigma_{33}(\omega_a)^2} {\sigma_{1}(0) }
= 2\frac{1}{ K_{a, c} C_a^2} ,
$$
or 
$$
     K_{a, c} = -\zeta_3^{-2c} \frac{2}{\sigma_{33}(\omega_a)^2C_a^2}.
$$
However $K_{a,c}$ satisfies 
$$
K_{a, 0} K_{a, 1} K_{a, 2} \equiv 1,
$$
and thus 
$$
	\sigma_{33}(\omega_a) = \frac{\sqrt{2}}{C_a},
$$
and $K_{a,c} = \zeta_3^c$.
\end{proof}

\section{A generalized Frobenius' theta formula}

In this section we give a 
generalized Frobenius theta formula, in analogy with 
$\sn^2(u) + \cn^2(u) = 1$, 
$\sn^2(u) + k^2 \dn^2(u) = 1$.

The second of our main theorems  is the following:
\begin{theorem} (A generalized Frobenius theta formula)
\label{th:Frob}
We have 
$$
\sum_{a=1}^4 \frac{\prod_{c=0}^2\al_a^{(c)}}{f'(b_a)} = 1,
$$
or
$$
2\sqrt{2}\sum_{a=1}^4 
\left(
\prod_{c=0}^2\frac{\sigma(u+\hzeta_3^c\omega_a)}
{\sigma(u)} \right)  = 1.
$$
\end{theorem}
This is a generalization of Corollary 7.5 of \cite[Ch.~III]{Mum}
which is a special case  of the Frobenius theta formula.

\begin{corollary} 
The cover of the Jacobian$\JJ^{\sharp}$ is embedded in
$\mathbb{P}^{12}$ as a subspace 
 satisfying the cubic relation,
$$
2\sqrt{2}\sum_{a=1}^4 
\left(
\prod_{c=0}^2\sigma(u+\hzeta_3^c\omega_a) \right)  = \sigma(u)^3.
$$
\end{corollary} 

For the proof of the theorem, we introduce
two quantities,
$$
	F(x) = F(P; P_1, P_2, P_3) := (x - x_1) (x - x_2) (x - x_3),
$$
$$
K((x, y)) :=
   \frac{\Bigr(\prod_{c=0}^2\mu((x,\zeta_3^c y); P_1, P_2, P_3)\Bigr) dx }
{3 f(x) F((x,y); P_1, P_2, P_3)} .
$$

\begin{lemma}
$K(P; P_1, P_2, P_3)$ does not vanish
for $P \to P_a$ $(a=1, 2, 3)$.
\end{lemma}

\begin{proof}
By letting $P \to P_1$ as
$(x, y) = (x_1 + t^3, y_1(1 + h(x_1) t))$,
we have
\begin{gather*}
\begin{split}
\mu(P; P_1, P_2, P_3)
& = t (C + d_{\ge}(t)),
\end{split}
\end{gather*}
whereas
\begin{gather*}
F(P; P_1, P_2, P_3)
 =  t^3 (K_4 + d_{\ge}(t)).
\end{gather*}
\end{proof}

Direct computations provide the following relations:
\begin{lemma}
\label{lm:resK}
$$
	\deg_{P=\infty} K = -1, \quad 
	\res_{P=\infty} K = -1,
$$
$$
	\deg_{P=(b_a, 0)} K = -1, \quad 
	\res_{P=(b_a, 0)} 
K = \frac{\prod_{c=0}^2\al_{a}^{(c)}(w(P_1, P_2, P_3))}{f'(b_a)}.
$$
\end{lemma}

\subsection*{Proof of Theorem \ref{th:Frob}}
Integrating over the sides of the polygon 
representation of $X$ gives:
$$
	\oint_\Gamma K = 0.
$$
Lemma \ref{lm:resK} provides the first relation
in Theorem \ref{th:Frob}. From Theorem 
\ref{th:al}, we have the relation of $\sigma$-functions.
\qed.

\section{Domain of the $\mathrm{al}$-function} \label{sec:Domain}


We give a domain to the trigonal
$\al_a^{(0)}$-function, in analogy to the fact that
 the hyperelliptic  $\al_a$ function is related to 
a Prym variety (see Appendix).
In fact, $\sqrt[3]{F_a}$ is a function on $S^3 \Gamma_\infty X$,
and more precisely  the domain of $\al_a^{(c)}$ is contained 
the preimage of $\JJ^{(a; c)}$ under an extended Abel map. 
We will also regard it as a subset of
 $S^3 \hat X$ for a suitable covering $\varpi:\hat X \to X$.
In this section, we construct $\hat X$.

When we consider the preimage of $\JJ^{(a; c)}$, 
we use a parameterization $z = \sqrt[3]{x - b_a}$.
As in the standard construction
of a Galois cover of a curve with $\mathbb{Z}_p$ action at a
given point $P$, cf. \cite{cornalba}, namely by normalizing the
curve obtained by taking the inverse image of the 1-section
under $\mathcal{L}\rightarrow\mathcal{L}^p$,
in the total space of a line bundle $\mathcal{L}$ such that
$\mathcal{L}^p=\mathcal{O}(P)$,
we consider
 a   
curve $\hat X$ whose affine representation can be given by
\begin{equation}
	w^3 = \prod_{i=1, \neq r}^{4} ( z^3 - a_i),
\label{w2:g}
\end{equation}
where $a_i = b_i - b_a$, $z = \sqrt[3]{x - b_a}$.
The birational change of variables  $w = y/z$
 transforms the plane curve $X$ to a
plane curve\footnote{ More precisely, we consider an extended ring
$R':=R[z]/(z^3-x+b_a)$ but since $\mathrm{Spec}(R')$ is
a singular curve, we normalize it to
$\hat R=\CC[w,z]/(w^3-\prod(z^3-a_i))$.
$\hat X$ is the projectivization of $\mathrm{Spec}(\hat R)$.
As in  (\ref{eq:zetaXinf}), we extend the $\ZZ_3$-action
consistently with  the Galois action on  $R$.}

$\hat X$ with 
the same $\ZZ_3$-action on $y$ and $z$.
If we adjoin the function $\chi=\sqrt{x-b_a}$
to the field of
meromorphic functions of $X$, the normalized curve is a space
curve $(3,8,13)$ \cite{KMP}, determined, e.g.,  by equations
 $y^3=\chi^2 \prod_{i=1, \neq r}^{4} ( \chi^2 - a'_i)$,
$\eta^3=\chi \prod_{i=1, \neq r}^{4} ( \chi^2 - a'_i)^2$.


On the other hand, for $c_i^3 := a_i$,
let $(\zeta_3^d c_i, 0) \in \hat X$, $(d = 0, 1, 2)$,
be 
a finite branch point $\hat B_i^{(d)}$  
and $(\zeta_3^d z, w) \in \hat X$, $(d = 0, 1, 2)$
a generic point $\hat P^{(d)}$.

There is an automorphism 
$$
\iota_{\zeta_3}: \hat X \to \hat X,  \quad((z, w) \mapsto (\zeta_3 z, w)),
$$
whereas there are
 the trigonal automorphisms
$$
\hzeta_3: \hat X \to \hat X\quad\mbox{and}\quad
\hzeta_3:  X \to X, 
\quad ((z, w) \mapsto (z, \zeta_3 w), \quad (x, y) \mapsto (x, \zeta_3 y)).
$$


These automorphisms generate distinct subgroups except at infinity.  


The  point at infinity of $\hat X$ is resolved into three points
$\infty^{(c)}$, $(c = 0, 1, 2)$.
At each $\infty^{(c)}$ of $\hat X$, 
$\hzeta_3$ and  $\iota_{\zeta_3}$ are identified, {\it i.e.},
\begin{equation}
\iota_{\zeta_3} : \infty^{(c)} \mapsto \infty^{(c+1\ \mbox{\tiny{mod}}\ 3)},
 \quad
\hzeta_3 :  \infty^{(c)} \mapsto  \infty^{(c+1\ \mbox{\tiny{mod}}\ 3)}.
\label{eq:zetaXinf}
\end{equation}
Also, $(0, 0) \in \hat X$ which corresponds to
$B_a \in  X$ is the fixed point of $\iota_{\zeta_3}$ and  $\hzeta_3$.

Then as illustrated in Figure 4, there is a {\it{trigonal covering}}:
$$
     \varpi: \hat X \to X, \quad 
      (\hat P=(z, w) \mapsto P = (z^3 + b_a, w z)),
$$
and 
$$
(0,0) \mapsto B_a, \quad
\hat B_i^{(d)} \mapsto B_i, \quad (i \in \{ 1, 2, 3, 4\} \setminus \{a\}).
$$

Henceforth we consider the $a = 1$ case without
loss of generality.
 From Figure 1 and Remark \ref{rmk:w''}, we have
$$
\Lambda_3^{(a; 0)} = 2 \ZZ \omega_1' + 6 \ZZ \omega_1''
+ 6 \ZZ \omega_2' + 2 \ZZ \omega_2''
+ 6 \ZZ \omega_3' + 2 \ZZ \omega_3''.
$$
Consistent with the covering map, the actions on 
$H_1(\hat X, \ZZ)$ are extended to $H_1(\hat X, \ZZ[\zeta_3])$.
Thus we have a natural lift of
the homology basis
$(\hat\alpha_1, \hat\beta_1, \hat\alpha_2^{(c)}, \hat\beta_2^{(c)},
\hat\alpha_3^{(c)}, \hat\beta_3^{(c)})$ 
$\in H_1(\hat X, \ZZ[\zeta_3])$
as  shown in Figure 3.
Here we have used the fact that
the actions of $\iota_{\zeta_3}$ and $\hzeta_3$ are exchanged at infinity 
due to the properties (\ref{eq:zetaXinf}). 
\begin{figure}
\begin{center}
\includegraphics[scale=0.5]{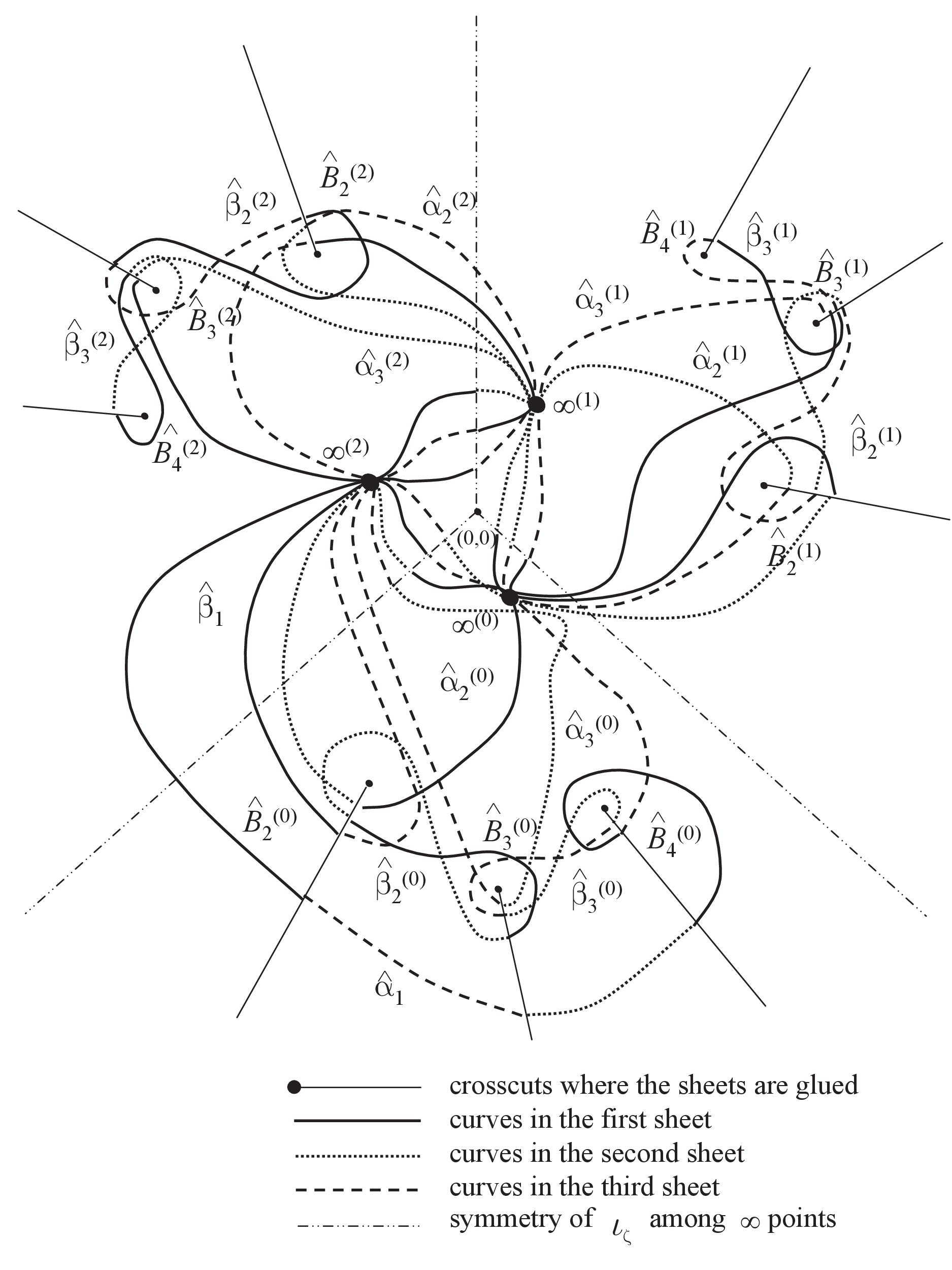}
\caption{Homology basis of $\hat X$}
\label{fig:3}
\end{center}
\end{figure}
\begin{figure}
\begin{center}
\includegraphics[scale=0.5]{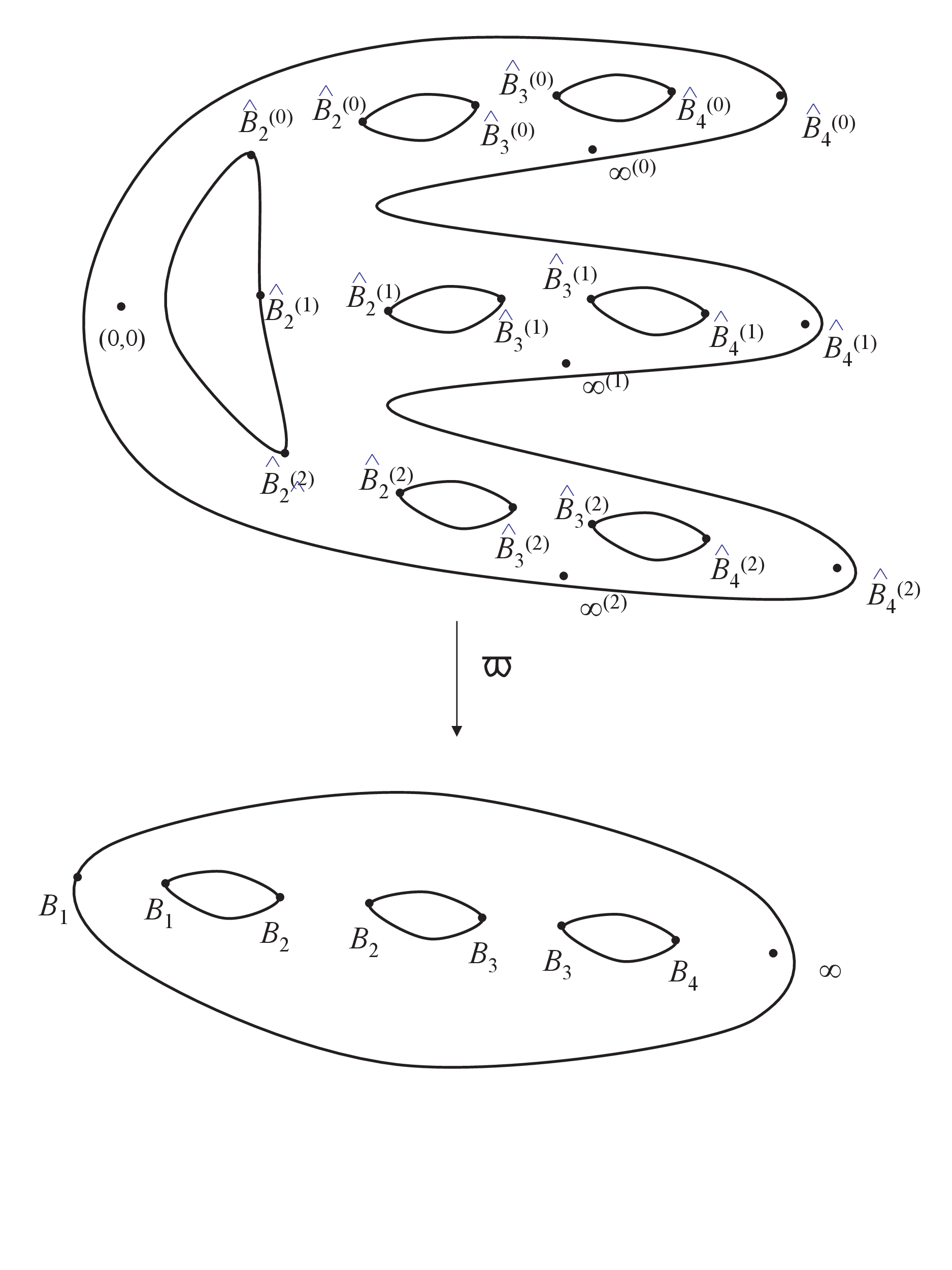}
\caption{$\varpi: \hat X \to X$}
\label{fig:4}
\end{center}
\end{figure}


Up to homotopy, $\hat \alpha_1$ equals
$\varpi \hat \alpha_1$;
similarly $\varpi \hat \beta_1 = \beta_1$.
On the other hand for $i = 2, 3$, Figure 3 shows that
\begin{eqnarray}
      &  \varpi(\alpha_i^{(c)})  = \alpha_i, \quad
        \varpi(\beta_i^{(c)})  = \beta_i, 
         \label{eq:gamma^(c)} \\
      &  \varpi^{-1}(\alpha_i) 
  =\hat \alpha_i^{(0)} + \hat \alpha_i^{(1)} + \hat \alpha_i^{(2)},
 \quad \varpi^{-1}(\beta_i) 
  =(\hat \beta_i^{(0)}, \hat \beta_i^{(1)}, \hat \beta_i^{(2)}).
  \nonumber 
\end{eqnarray}

Let the Jacobian associated with $\hat X$ be denoted by
$\hat {\JJ}_7$ and a basis of holomorphic one-forms be
 given by\footnote{
For a curve of genus 7 $X_7$ whose affine part is given by:
$w^3 = \prod_{i=0}^8(z-\hat c_i)$, 
letting
$$
z' := \frac{1}{z - c_0}, \quad
w':=\frac{1}{\sqrt[3]{\prod_{i=1}^8(c_i - c_0)}}
     \frac{w}{(z-c_0)^3},\quad
c_i' := \frac{1}{c_i - c_0}; \quad
$$
as another affine chart, cf.  \cite[Ch.~IIIa]{Mum} and \cite[Appendix]{KMP},
we have
$ {w'}^3 = \prod_{i=1}^8(z'-c'_i)$.
The holomorphic one-forms are given by
$$
     \hat \nuI'_1 = \frac{dz'}{3 {w'}^2}  , \quad
     \hat \nuI'_2 = \frac{z' dz'}{3 {w'}^2}, \quad
     \hat \nuI'_3 = \frac{{z'}^2 dz'}{3 {w'}^2}, \quad
     \hat \nuI'_4 = \frac{dz'}{3 w}, \quad
     \hat \nuI'_5 = \frac{{z'}^3 dz'}{3 {w'}^2}, \quad
     \hat \nuI'_6 = \frac{z'w' dz'}{3 {w'}^2}, \quad
     \hat \nuI'_7 = \frac{{z'}^4 dz'}{3 {w'}^2}. \quad
$$
We state that when $c_0 =0$, $\hat \nuI_i =-3 \hat \nuI'_{8-i}$
denoting $d z' = - dz / z^2$.
}
$$
     \hat \nuI_1 = \frac{dz}{w^2}, \quad
     \hat \nuI_2 = \frac{z dz}{w^2}, \quad
     \hat \nuI_3 = \frac{z^2 dz}{w^2}, \quad
     \hat \nuI_4 = \frac{dz}{w}, \quad
$$
$$
     \hat \nuI_5 = \frac{z^3 dz}{w^2}, \quad
     \hat \nuI_6 = \frac{zw dz}{w^2}, \quad
     \hat \nuI_7 = \frac{z^4 dz}{w^2}. \quad
$$
We have removed the factor $3$ in the denominators for
later convenience.
The Abel map 
$\hat w:\Gamma_{\infty^{(0)}} \hat X \to \CC^7$
defined by these holomorphic one forms is denoted by
$$
	\hat w(\Gamma_{P_1,{\infty^{(0)}}}, \ldots, \Gamma_{P_n,{\infty^{(0)}}}) 
= \sum_{i=1}^n \hat w(\Gamma_{P_i,{\infty^{(0)}}}),
        \quad
        \hat w(\Gamma_{P_i,{\infty^{(0)}}}) = \int^{\Gamma_{P_i,{\infty^{(0)}}}}\hat \nuI.
$$

Since $d x = 3 z^2 dz$,
we have the relation
$$
     \hat \nuI_1 =  \varpi^*\nuI_1, \quad
     \hat \nuI_5 = \varpi^*\nuI_2, \quad
     \hat \nuI_6 =  \varpi^*\nuI_3, \quad
     \varpi^*(\nuI) =  
     \begin{pmatrix} 
     \hat \nuI_1\\
     \hat \nuI_5\\
     \hat \nuI_6\\
     \end{pmatrix}.
$$

However we should note the $\hzeta_3$-action
 according to the covering defined by $z$:


\begin{lemma} \label{lm:10a}
The maps $\iota_{\zeta_3}$ and $\hzeta_3$ induce
$$
     \iota_{\zeta_3}\varpi^*(\nuI) = 
     \varpi^*(\hat\zeta_3\nuI) = 
     \hat\zeta_3\varpi^*(\nuI).
$$
\end{lemma}

We consider the projections of the extended Abel maps,
$$
   \hat w(\Gamma_{\hat P,\infty^{(0)}})|_{\varpi^* \nuI} 
= \int_{\Gamma_{\hat P,\infty^{(0)}}} \varpi^*\nuI,
$$
and  we have the Lemma:

\begin{lemma} \label{lm:10b}
$$
 \int_{\gamma} \nu^I
	=\frac{1}{3} \int_{\hat \gamma} \varpi^*\nuI, \quad
   \hat\gamma = \varpi^{-1}  \gamma,\quad
     \gamma = \beta_1, \alpha_i, (i = 2, 3),
$$
$$
 \int_{\gamma} \nu^I = \int_{\hat \gamma} \varpi^*\nuI, \quad
   \gamma = \varpi \hat \gamma,\quad
    \hat \gamma = \alpha_1, \alpha_i^{(c)}, \beta_i^{(c)},
   (i = 2, 3, c = 0, 1, 2).
$$
\end{lemma}
\begin{proof}
For a neighborhood $U$ of a point in $\alpha_i$, 
we have $U^{(c)}$ $(c=0, 1, 2)$ such that $U = \varpi U^{(c)}$.
Then $\varpi^* \nuI|_{U^{(c)}} = \iota_{\zeta_3}^c \varpi^* \nuI|_{U^{(0)}}$
and they cancel the phase difference between 
$\alpha_i^{(c)}$ and $\alpha_i^{(0)}$.
Then we have
\begin{gather*}
\begin{split}
\int_{\hat\beta_1} \varpi^*\nuI
&=\left(\int_{\hat B_2^{(0)}}^{(0,0)}+\int_{(0,0)}^{\hat B_2^{(1)}}
+\int_{\hat B_2^{(1)}}^{(0,0)}+\int_{(0,0)}^{\hat B_2^{(2)}}
+\int_{\hat B_2^{(2)}}^{(0,0)}+\int_{(0,0)}^{\hat B_2^{(0)}}\right)
\varpi^*\nuI \\ 
&=\left(
\int_{\beta_1}
+\hzeta_3^2\int_{\hzeta_3\beta_1}
+\hzeta_3\int_{\hzeta_3^2\beta_1} \right) \nuI  =3 \int_{\beta_1} \nuI. \\
\end{split}
\end{gather*}
Here we used the fact that $\zeta_3 + \zeta_3^2 = -1$. Similarly
\begin{gather*}
\begin{split}
\int_{\varpi^{-1}\alpha_i} \varpi^*\nuI
&= \int_{\alpha_i^{(0)} + \alpha_i^{(1)} + \alpha_i^{(2)}} \varpi^*\nuI 
 =3 \int_{\alpha_i} \nuI. \\
\end{split}
\end{gather*}
\end{proof}

Then we have the period matrices for $\hat X$,
$$
  (\hat \omega') := 
   \frac{1}{2}(
\int_{\hat \alpha_1}  \hat \nuI , 
\int_{\hat \alpha_2^{(0)}}  \hat \nuI , 
\int_{\hat \alpha_2^{(1)}}  \hat \nuI , 
\int_{\hat \alpha_2^{(2)}}  \hat \nuI , 
\int_{\hat \alpha_3^{(0)}}  \hat \nuI , 
\int_{\hat \alpha_3^{(1)}}  \hat \nuI , 
\int_{\hat \alpha_3^{(2)}}  \hat \nuI ),
$$
$$
  (\hat \omega'') := 
   \frac{1}{2}(
\int_{\hat \beta_1}  \hat \nuI , 
\int_{\hat \beta_2^{(0)}}  \hat \nuI , 
\int_{\hat \beta_2^{(1)}}  \hat \nuI , 
\int_{\hat \beta_2^{(2)}}  \hat \nuI , 
\int_{\hat \beta_3^{(0)}}  \hat \nuI , 
\int_{\hat \beta_3^{(1)}}  \hat \nuI , 
\int_{\hat \beta_3^{(2)}}  \hat \nuI ).
$$
Using these, we have the lattice $\hat \Lambda$ in $\CC^7$ and
 the Jacobian $\hat \JJ_7$  is:
$\hat \JJ_7 = \CC^7/ \hat\Lambda$.

Lemmas \ref{lm:10a} and \ref{lm:10b} give the following
proposition:

\begin{proposition}
$$
 (\hat\omega', \hat\omega'')|_{\varpi^* \nuI}
  =(\omega_1',
  \omega_2',
  \omega_2',
  \omega_2',
  \omega_3',
  \omega_3',
  \omega_3',
  3\omega_1'',
  \omega_2'',
  \omega_2'',
  \omega_2'',
  \omega_3'',
  \omega_3'',
  \omega_3'').
$$
\end{proposition}
In consequence, it is natural to introduce 
$$
\hat \Lambda^{\al;(1,0)}:= \sum_{i=1}^3
  \left(\ZZ \hat \omega'_{\al^{(c)}, i} 
+\ZZ \hat \omega''_{\al^{(0)}, i} \right),
$$
where
$$
  (\hat \omega'_{\al^{(1,c)}, i})_{i=1,2,3} := 
   \frac{1}{2}(
\int_{\hat \alpha_1}  \varpi^* \nuI , 
\int_{\hat \alpha_2^{(0)} +\hat \alpha_2^{(1)} +\hat \alpha_2^{(2)} }
  \varpi^* \nuI,
\int_{\hat \alpha_3^{(0)} +\hat \alpha_3^{(1)} +\hat \alpha_3^{(2)} }
  \varpi^* \nuI),
$$
$$
  (\hat \omega''_{\al^{(1,0)}, i})_{i=1,2,3} := 
   \frac{1}{2}(
\int_{\hat \beta_1}  \varpi^* \nuI , 
\int_{\hat \beta_2^{(c)}} \iota_{\zeta_3}^c \varpi^* \nuI , 
\int_{\hat \beta_3^{(c)}} \iota_{\zeta_3}^c \varpi^* \nuI ) .
$$

On the other hand, we introduce a natural subvariety of
$\hat \JJ_7$ as follows:
\begin{definition}
For the projection of the extended Abel map,
$$
   \hat w|_{\varpi^* \nuI}: S^3 \Gamma_{\infty^{(0)}}\hat X \to \CC^3 \subset 
       \CC^7,
$$
${\hat\JJ}^{\al;(1,0)}$ is defined by
$$
{\hat\JJ}^{\al;(1,0)} := 
\frac{\hat w(S^3 \Gamma_{\infty^{(0)}}\hat X)|_{\varpi^* \nuI} }{
            \hat \Lambda_{\al^{(1,0)}}},
$$
for $c = 0, 1, 2$.
\end{definition}

Then ${\hat\JJ}^{\al;(1,0)} \subset \hat {\JJ}$
 is the domain of the $\al_1$-function,
$$
\hat\al_1(u) := \frac{1}{z_1 z_2 z_3} 
\frac{\left| \begin{matrix}
1 & z_1^3 & w_1z_1 & z_1^6 \\
1 & z_2^3 & w_2z_2 & z_2^6 \\
1 & z_3^3 & w_3z_3 & z_3^6 \\
1 & b_a   & 0  & b_a^2   \\
\end{matrix} \right|}
{\left| \begin{matrix}
1 & z_1^3 & w_1z_1 \\
1 & z_2^3 & w_2z_2 \\
1 & z_3^3 & w_3z_3 \\
1 & b_a   & 0    \\
\end{matrix} \right|},
$$
where $u = \sum_{i=1}^3\hat w((z_i, w_i))|_{\varpi^* \nuI}$
for $((z_i, w_i))_{i=1, 2, 3} \in S^3 \hat X$.

Finally we have the following proposition.
\begin{proposition}
\begin{enumerate}
\item 
 $\hat w((z_i, w_i))|_{\varpi^* \nuI}$ is surjective.

\item 
There is an equality
$$
\hat\al_1(u) = \varpi^* \al_1^{(0)}(u).
$$ 

\item
By identifying $\hat w(S^3 \hat X)|_{\varpi^* \nuI} =\CC^3$ with 
$w(S^3 X) =\CC^3$,
${\hat\JJ}^{\al,(1,0)}$ agrees with ${\hat\JJ}^{(1;0)}$.
\end{enumerate}
\end{proposition}

\setcounter{section}{0}
\renewcommand{\thesection}{\Alph{section}}
\section{Appendix: Hyperelliptic $\al$ Functions}

In this appendix, we review the hyperelliptic $\al$-function
mainly following \cite{Ba1, Ba2}.

{\textit{ Hyperelliptic Curve:}}
We let  a (hyper)elliptic curve $C_g$  of genus $g$
$(g>0)$ be defined by the affine equation,
\begin{gather} \split
   y^2 &= (x-b_0)(x-b_1)(x-b_2)\cdots (x-b_{2g})\\
       &= P(x) Q(x),
\endsplit  \label{H2-1}
\end{gather}
where $b_j$'s are distinct complex numbers,
$P(x) = (x-b_1)(x-b_3) \cdots (x-b_{2g-1})$ and
$Q(x) := y^2 / P(x)$. Let $(b_j,0) = B_j \in C_g$.

For a point $(x, y)\in C_g$,
  differentials of the first kind (not normalized in the standard
way which gives the identity as the matrix of $A$-periods)
are
defined by,
\begin{gather*}   
     \nuI_i :=  \frac{x^{i-1} d x}{2y}.  \label{H2-3}
\end{gather*}
The extended Abel map from the $g$-th symmetric product
of the universal cover  $\Gamma_\infty C_g$ 
of the curve $C_g$ to $\CC^g$ is defined by,
\begin{gather*}
 w :{S}^g \Gamma_\infty C_g \longrightarrow \mathbb C^g,
\quad
      \left( w(\Gamma_{(x_1,y_1),\infty},\ldots,
      \Gamma_{(x_g,y_g),\infty})
:= \sum_{i=1}^g
       \int_{ \Gamma_{(x_g,y_g),\infty}} \nuI \right),
\end{gather*}
where $\Gamma_{(x_g,y_g),\infty}$ is a path in  the path space
$\Gamma_\infty C_g$.

Consider
$
\mathrm{H}_1(C_g, \mathbb Z)
  =\bigoplus_{j=1}^g\mathbb Z\alpha_{j}
   \oplus\bigoplus_{j=1}^g\mathbb Z\beta_{j},
$  the homology group of the hyperelliptic curve $C_g$,
where the intersections are given by
$[\alpha_i, \alpha_j]=0$, $[\beta_i, \beta_j]=0$ and
$[\alpha_i, \beta_j]=\delta_{i,j}$.
Here we employ the choice illustrated in
Figure 5.
\begin{figure}
\begin{center}
\includegraphics[scale=0.5]{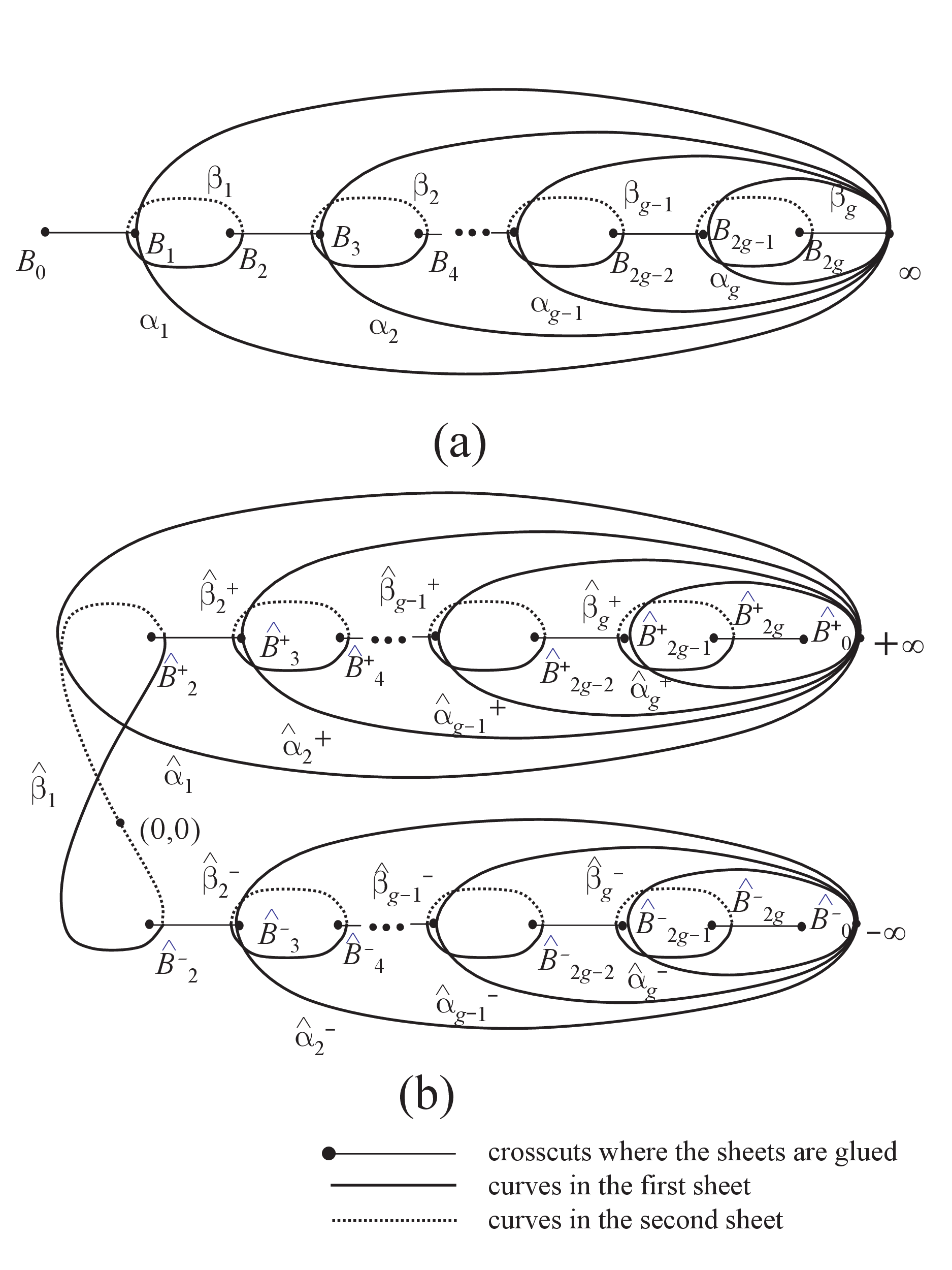}
\caption{(a): $C_g$ and (b): $\hat C_{2g-1}$}
\label{fig:5}
\end{center}
\end{figure}
The (half-period) hyperelliptic integrals
of the first kind are defined by,
\begin{gather*}    {\omega}':=\frac{1}{2}\left[\left(
     \int_{\alpha_{j}}\nuI^{}_{i}\right)_{ij}\right],
\quad
      {\omega}'':=\frac{1}{2}\left[\left(
       \int_{\beta_{j}}\nuI^{}_{i}\right)_{ij}\right],
 \quad
    {\omega}:=\left[\begin{matrix} {\omega}' \\ {\omega}''
     \end{matrix}\right].
  \label{H2-9}
\end{gather*}
If we let:
$$
	\omega_a:=\int^{B_a}_\infty \nuI, \quad
        (a=0, 1, 2, \cdots, 2g-1, 2g),
$$
Figure 5 shows:
$$
	\omega'_a = \omega_{2a - 1}, \quad
	\omega''_a = \omega_{2a} - \omega_{2a - 1}, \quad
         a>1.
$$
The Jacobian $\JJ_g$
is defined as the complex torus,
\begin{gather*}
   \JJ_g := \mathbb C^g /{{\Lambda}}_g.
     \label{H2-10}
\end{gather*}
Here  ${\Lambda}_g$   is a  $2g$-dimensional
lattice generated by the period matrix given by $2{\omega}$.
We also use the same letter $u$ for a vector in  $\mathbb C^g$
and a point of the Jacobian
$\JJ_g$.

Using the (unnormalized) differentials of the second kind,
\begin{gather*}
     \nuII_j =\dfrac{1}{2 y}\sum_{k=j}^{2g-j}(k+1-j)
      \lambda_{k+1+j} x^k d x ,
     \quad (j=1, \ldots, g), \label{H2-38}
\end{gather*}
the half-period hyperelliptic  matrices
of the second kind are defined by,
\begin{gather*}    {\eta}':=\frac{1}{2}\left[\left(
         \int_{\alpha_{j}}\nuII_{i}\right)_{ij}\right],
\quad
      {\eta}'':=\frac{1}{2}\left[\left(
        \int_{\beta_{j}}\nuII_{i}\right)_{ij}\right].
  \label{H2-39}
\end{gather*}
The hyperelliptic $\sigma$ function,
which is a holomorphic
function over $u\in \mathbb C^g$, is defined by
[\cite{Ba2}, p.336, p.350], \cite{Kl1, BEL2},
\begin{gather} \sigma(u):=\sigma(u;C_g):
  =\ \gamma\mathrm{exp}(-\dfrac{1}{2}\ ^t\ u
  {\eta}'{{\omega}'}^{-1}u)
  \vartheta\negthinspace
  \left[\begin{matrix} \delta'' \\ \delta' \end{matrix}\right]
  (\frac{1}{2}{{\omega}'}^{-1}u ; \tau),
     \label{H2-40}
\end{gather}
where $\gamma$ is a certain constant factor,
$\vartheta\left[\negthinspace \right]$ is the Riemann $\theta$ function
with characteristics,
\begin{gather*}
\vartheta\negthinspace\left[\begin{matrix} a \\ b
 \end{matrix}\right]
     (z;  \tau)
    :=\sum_{n \in \mathbb Z^g} \exp \left[2\pi \sqrt{-1}\left\{
    \dfrac 12 \ ^t\negthinspace (n+a) \tau(n+a)
    + \ ^t\negthinspace (n+a)(z+b)\right\}\right],
     \label{H2-41}
\end{gather*}
with $ \tau:={{\omega}'}^{-1}{\omega}''$
for $g$-dimensional vectors $a$ and $b$,
and
\begin{gather*}
 \delta' :=\ ^t\left[\begin{matrix} \dfrac {g}{2} & \dfrac{g-1}{2}
       & \cdots
      & \dfrac {1}{2}\end{matrix}\right],
   \quad \delta'':=\ ^t\left[\begin{matrix} \dfrac{1}{2} & \cdots
& \dfrac{1}{2}
   \end{matrix}\right].
     \label{H2-42}
\end{gather*}

\begin{proposition} 
\label{Hprop:sigma} 
If for $u$, $v\in\CC^3$, and $\ell$
($=2\omega'\ell'+2\omega''\ell''$) $\in\Lambda$, we define
\begin{align*}
  L(u,v)    &:=2{}^t{u}(\eta'v'+\eta''v''),\nonumber \\
  \chi(\ell)&:=\exp[\pi\sqrt{-1}\big(2({}^t {\ell'}\delta''-{}^t
  {\ell''}\delta') +{}^t {\ell'}\ell''\big)] \ (\in \{1,\,-1\}),
\end{align*}
the following holds
\begin{equation*}
	\sigma(u + \ell) = \sigma(u)
               \exp(L(u + \frac{1}{2}\ell, \ell)) \chi(\ell).
        \label{Heq:4.11}
\end{equation*}
\end{proposition}

\bigskip
\begin{definition}\label{Hdef-2.3}
\begin{enumerate}[{(}1{)}]
\item We define the double coverings of $\JJ_g$ by
$$
	\JJ_g^{(a)} = \CC^g / \Lambda^{(a)},
$$
where $\Lambda^{(0)} :=  \bigcap_{a=1}^{2g} \Lambda^{(a)}$,
\begin{gather*}
\begin{split}
      \Lambda^{(a)} :=&  2 \ZZ \omega'_a + 4 \ZZ \omega''_a +
      \sum_{b=1,\neq a}  (2 \ZZ \omega'_b + 2 \ZZ \omega''_b) 
      \quad \mbox{for }  a = 1, 3, \ldots, 2g-1,\\
      \Lambda^{(a)} :=&  
4 \ZZ \omega'_a +4 \ZZ \omega''_a) +
      \sum_{b=1,\neq a}  (2 \ZZ \omega'_b + 2 \ZZ \omega''_b) 
      \quad \mbox{for }  a = 2, 4, \ldots, 2g.\\
\end{split}
\end{gather*}

\item 
For a point
$\Gamma_{P_c,\infty}\in \Gamma_\infty C_g$,
$$
\varepsilon_r^{(c)}: \Gamma_{\infty} C_g \to \ZZ_2
$$
 be defined by $\varepsilon_r^{(c)}:=w_r-w_\infty$
for the winding number $w_r$ around $B_a$ in 
$\kappa_\infty\Gamma_\infty C_g$ and
 the winding number $w_\infty$ around $\infty$ in 
$\kappa_\infty\Gamma_\infty C_g$.
For a point
$(\Gamma_{P_1,\infty}, \Gamma_{P_2,\infty},\ldots ,\Gamma_{P_g,\infty})$ 
in $S^g \Gamma_{\infty} C_g$,
let
$$
\varepsilon_r: S^g \Gamma_{\infty} C_g \to \ZZ_2, \quad
(\varepsilon_r:= \varepsilon_r^{(1)} + \varepsilon_r^{(2)}
                + \cdots +\varepsilon_r^{(g)}). 
$$

\item 
For a point
$(\Gamma_{P_1,\infty}, \Gamma_{P_2,\infty},\ldots ,\Gamma_{P_g,\infty})$ 
in $S^g \Gamma_{\infty} C_g$, let $u=w
(\Gamma_{P_1,\infty}, \Gamma_{P_2,\infty},\ldots ,\Gamma_{P_g,\infty})$. 
The hyperelliptic $\mathrm{al}$ function over
$\JJ^{(r)}$ and $w^{-1}\JJ^{(r)}$ as a subset of a quotient space of
in $S^g \Gamma_{\infty} C_g$,
is formally defined by
\cite[p.340]{Ba2}, \cite{Wei},
\begin{gather}
\mathrm{al}_r(u): = (-1)^{\varepsilon_r(
\Gamma_{\infty, P_1}, \Gamma_{\infty, P_2},\ldots ,\Gamma_{\infty, P_g})}
\sqrt{F(b_r)} , \label{H2-20}
\end{gather}
where 
\begin{gather}
	F(x):= (x-x_1) \cdots (x-x_g),
          \label{H2-21}
\end{gather}
for a preimage $(\Gamma_{(x_i, y_i),\infty})_{i=1, \ldots, g}
\in S^g\Gamma_{\infty}C_g$ of
$w((\Gamma_{(x_i, y_i),\infty})_{i=1, \ldots, g}) = u \in \JJ^{(r)}$ 
under the Abel map.
\end{enumerate}
\end{definition}


\begin{remark}
{\rm{
The definition (\ref{H2-20}) is historically
\begin{gather}
\mathrm{al}_r(u) = 
\tilde \gamma_r\sqrt{F(b_r)} ,  \label{H2-new}
\end{gather}
where $\tilde \gamma_r:=\sqrt{-1/P'(b_r)}$.
Thus the preimage of $w$ of $\JJ^{(r)}$ is a quotient space of
$S^g\Gamma_{\infty} C_g$.
We comment on the sign $(-1)^{\varepsilon}$ in
 the right-hand side of (\ref{H2-20}).
The hyperelliptic curve $C_g$ admits the hyperelliptic
involution $\iota_H : (x, y) \to (x, -y)$. In a neighborhood 
of the branch point $B_r=(b_r,0)$, $y$ or $t$ such that
 $t^2 = (x - b_r)$
are local parameters. Thus for $t_i$ such that $t_i^2 := (x_i - b_r)$
$\iota_H^{(a)} t_i = - t_i$.
Similarly, $t_1 t_2 \cdots t_g$ is  defined
in a neighborhood of $B_r$ and 
$\iota_H$ can be made to act on the product:
a circuit  around the point produces the factor
$(-1)^{\varepsilon_r(
\Gamma_{\infty, P_1}, \Gamma_{\infty, P_2},\ldots ,\Gamma_{\infty, P_g})}$.

Further  the inverse $1/t_i$
is a local parameter at $\infty$ and thus
there is an action $\iota_H^{(\infty)} (1/t_i) = - (1/t_i)$,
and a circuit around $\infty$  generates
$(-1)^{\varepsilon_r(
\Gamma_{\infty, P_1}, \Gamma_{\infty, P_2},\ldots ,\Gamma_{\infty, P_g})}$.

However we claim that we can make sense of $t_1 t_2 \cdots t_g$ globally and 
  (\ref{H2-new}) holds globally by (\ref{H2-20}).  
In analogy to Jacobi's sn, cn, dn functions,
we need to extend the domain of the Jacobi inversion from $\JJ_g$
 to $\JJ_g^{(r)}$ and  $\JJ_g^{(0)}$. We show the extension in 
Proposition \ref{prop:Jal}; here
we  consider the behavior of the right-hand side of 
(\ref{H2-20}). Let us regard it as a function of $w(P_1)$
by fixing $P_2$, $\ldots$, $P_g$. 
Then a circuit around  $\alpha_b$ (see Figure 5 (a)) does not
have any effect on the sign factor of $t_1$.
On the other hand, when we go around $\beta_a$ in Figure 5 (a) once,
 $t_1$ acquires a sign and in order to cancel it,
we need to go twice  around $\beta_a$.
Thus the (homotopy) equivalence relation
 is the same as that which holds for $\JJ_g^{(a)}$.
}}
\end{remark}

\bigskip

\begin{proposition}
Introducing the half-period $\omega_r :=
\int^{b_r}_\infty du^{}$,
we have the relation {\rm\cite[340]{Ba2}},
\begin{gather}
	\al_r(u) =\gamma_r'' \frac{
\exp(-{}^t u  \varphi_r)
\sigma( u + \omega_r)}{\sigma(u)},
\quad r=1,2,\ldots,2g,
           \label{H2-43}
\end{gather}
where $\gamma_r''$ is a certain constant.
$$
     \varphi_r =
 \left\{\begin{matrix} 
{\eta}'{{\omega}'}^{-1}\omega'_r & r = 1, 3, \ldots, 2g - 1, \\
{\eta}''{{\omega}''}^{-1}\omega''_r
+{\eta}'{{\omega}'}^{-1}\omega'_r
 & r = 2, 4, \ldots, 2g . \\
\end{matrix}\right.
$$
\end{proposition}

\begin{proof}
By comparing zeros and poles of both sides, we have the
result.   \end{proof}



\begin{proposition} \label{Hprop:al_period}
For a lattice point $\ell$ in $\Lambda^{(b)}$
$$
\al_b(u)=\al_b(u+ \ell).
$$
\end{proposition}


\begin{proof} We know:
$$
 \frac{\sigma(u + \omega_b+ \ell)}{\sigma(u+\ell)} = 
\frac{\sigma(u + \omega_b)}{\sigma(u)}
 \exp(L(\omega_b , \ell)) .
$$
For the $b = 2a -1$ case, 
\begin{gather}
\begin{split}
L( \omega_b' , \ell) &= 2^t\omega_b' (\eta'\ell' +\eta''\ell'')\\
 &=2^t\omega_b' \eta'\ell' +2^t\omega_b''\eta'\ell''-\pi\ii\ell''_b,\\
\end{split}
\end{gather}
whereas
\begin{gather}
 2^t(\omega'\ell' +\omega''\ell'')\eta' \omega^{\prime -1}\omega_b'
 = 2^t\ell'{}^t\eta'\omega_b'+2{}^t\ell''\omega_b''\eta'.
\end{gather}
Hence we have the equality.
\end{proof}

As a generalization of the relation $\sn^2 u + \cn^2 u = 1$,
we have the following relation.
\begin{proposition}
Let $A_a(x) = P(x) (x-b_a)$ and $a \in \{2, 4, \ldots, 2g\}$.
$$
    \sum_{r = 1, 3, \ldots, 2g-1, a} \frac{\al_r(x)^2}{A'_a(b_r)} =1.
$$
\end{proposition}
\begin{proof} See \cite[p.292]{Wei} and also \cite[Proposition 3.4]{Ma3}.
\end{proof}
\begin{remark}{\rm{
The relation implies the $g$ homogeneous identities,
$$
\sum_{r = 1, 3, \ldots, 2g-1, a}
\frac{(\gamma_r'')^2 \ee^{-2{}^t u  \varphi_r} }{A'_a(b_r)} 
\sigma( u + \omega_r)^2
\equiv\sigma(u)^2, \quad a = 2, 4, \ldots, 2g,
$$
among  $2g+1$ homogeneous coordinates, namely,
 $\sigma( u + \omega_r)$ ($r=1,2,\ldots,2g$) and $\sigma(u)$.
Noting that the square of each $\al_r$ is a function over
the hyperelliptic Jacobi variety $\JJ_g$,
these quadrics cut out the image of the Jacobian,
which is a $g$-dimensional variety
embedded in $\mathbb{P}^{2g}$.
}}
\end{remark}

\begin{remark}{\rm{

For the genus-one case,
the Weierstrass $\wp$ function corresponds to a curve
$y^2 =  (x - e_1) (x - e_2) (x - e_3)$,
whereas the Jacobi $\sn$ functions is defined on:
\begin{equation}
           w^2 = (z^2 - 1) (z^2 - k^2),
\label{Hw2:1}
\end{equation}
where $w = y /  z\sqrt{(e_2-e_1)^3}$, $z = \sqrt{(x - e_1)/(e_2-e_1)}$ and
$$
	\frac{dx}{2 y} =2 \sqrt{e_2-e_1} \frac{ dz }{2 w}.
$$
We have employed a curve (\ref{H2-1}) with $f(x)$ of odd
degree (thus a branchpoint at $\infty$),
 and the associated $\wp_{ij}$ function.

Note that when $g=1$, (\ref{Hw2:g}) is essentially reduced to
(\ref{Hw2:1}).
}}
\end{remark}

Given that the $\al_r$ function is a generalization of the $\sn$-function,
 we considered
a  genus $2g-1$  curve $\hat C_{2g-1}$ whose affine part is given by
\begin{equation}
	w^2 = \prod_{i=1, \neq r}^{2g+1} ( z^2 - a_i),
\label{Hw2:g}
\end{equation}
where $a_i = b_i - b_r$, $z = \sqrt{x - b_r}$, and $w = y/z$.

Let $(b_i, 0)=B_i \in C_g$ be the branch points on the affine plane
and $(x, y) \in C_g$ be a general point $P$.
For $c_i^2 := a_i$,
let $(\pm c_i, 0) \in \hat C_{2g-1}$
be $\hat B_i^\pm$ as a finite branch point
and $( z,  w) \in C_g$ be 
a general point $\hat P^\pm$.

There is an involution $\iota_A: (z, w) \mapsto (- z, w)$
as well as
 the hyperelliptic involution $\hat\iota_H: (z, w) \mapsto (z, - w)$
and $\iota_H: (x, y) \mapsto (x, -y)$.

At the point $\infty$ of $\hat C_{2g+1}$, acting by $\iota_H$
and $\iota_A$, we identify the actions
$\hat\iota_H$ and  $\iota_A$, i.e.,
$$
\hat\iota_H : \pm \infty \mapsto \mp \infty, \quad
\iota_A : \pm \infty \mapsto \mp \infty.
$$
On the other hand $(0, 0) \in \hat C_{2g+1}$, which corresponds to
$B_r \in  C_{g}$ is the fixed point of $\hat\iota_H$ and  $\iota_A$.

Let us consider the $r=1$ case.
Then there is a {\it{double covering}}:
$$
     \varpi_{g}: \hat C_{2g-1} \to C_g, \quad 
     (\hat P =(z, w) \mapsto P = (z^2 + b_r, wz)),
$$
and
$$
(0,0) \mapsto B_1, \quad
	\hat B_i^\pm \to B_i, \quad (i = 2, 3, \cdots, 2g, 2g+1).
$$
We illustrate this in Figure 6,
which is essentially the same as the picture in \cite[p.296]{ACGH}.
\begin{figure}
\begin{center}
\includegraphics[scale=0.5]{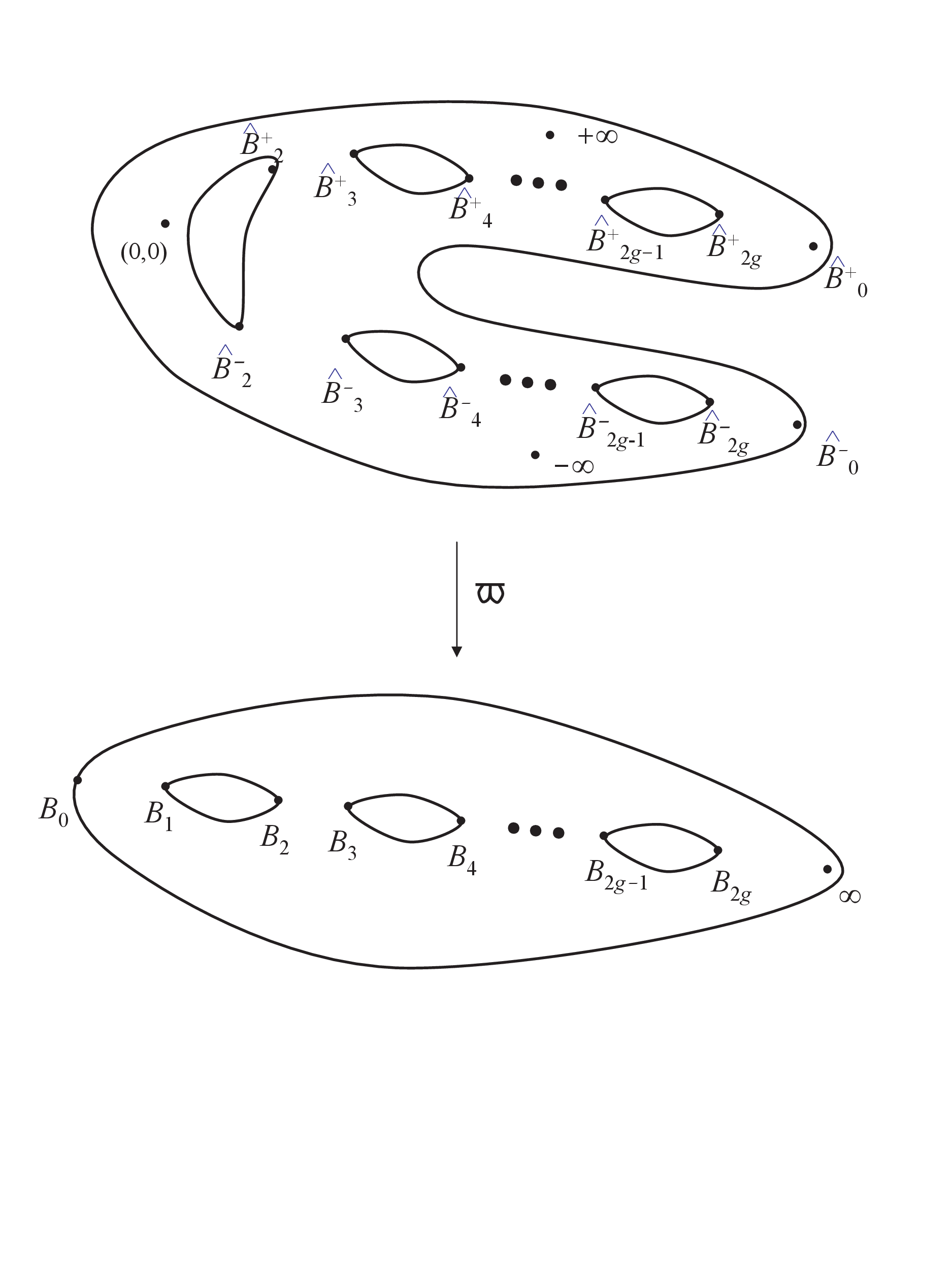}
\caption{$\varpi: \hat C_{2g-1} \to C_g$}
\label{fig:6}
\end{center}
\end{figure}

The (unnormalized) basis of  holomorphic one-forms
over $\hat C_{2g-1}$ is denoted by
$$
     \hat \nuI := \begin{pmatrix}
     \hat \nuI_1 \\
     \hat \nuI_2 \\
        \vdots\\
     \hat \nuI_{2g-1} \\
      \end{pmatrix}, \quad
     \hat \nu^I_j = \frac{z^{j-1} dz}{ w},
        \quad (j = 1, 2, \ldots, 2g - 1).
$$
Here we have removed the factor $1/2$ for later convenience.
Let us consider the Abel map
\begin{gather*}
\hat w :{S}^k \Gamma_{-\infty} \hat C_{2g-1} \longrightarrow \CC^{2g-1},
\quad
      \left( \hat w(\Gamma_{(x_1,y_1),-\infty},\ldots,\Gamma_{(x_k,y_k),-\infty}):= \sum_{i=1}^k
       \int_{\Gamma_{(x_i,y_i),-\infty}} \hat\nuI \right).
\end{gather*}
As the contours in Figure 5 (b) illustrate,
the associated periodic matrices are given as,
$$
  (\hat \omega', \hat \omega'') := 
   \frac{1}{2}\left(
\left(\int_{\hat \alpha_1} \hat \nuI, 
\left( \int_{\hat \alpha_i^+} \hat \nuI, 
\int_{\hat \alpha_i^-} \hat \nuI\right)_{i=2, \ldots, g}\right),
\left(\int_{\hat \beta_1} \hat \nuI, 
\left( \int_{\hat \beta_i^+} \hat \nuI, 
\int_{\hat \beta_i^-} \hat \nuI\right)_{i=2, \ldots, g}\right)\right).
$$
The lattice associated with the curve $\hat C_{2g-1}$ is
denoted by $\hat \Lambda$ and its Jacobian  by
$\hat\JJ_{2g-1} = \CC^{2g-1}/\hat\Lambda$.

Direct computations show the following facts:
\begin{proposition}
\begin{enumerate}
\item
$$
	\frac{z^{2i-2} dz}{w} = \frac{x^{i-1} dx}{2 y},
           \quad (i=1, \ldots, g),
\quad \varpi^*\nuI = \begin{pmatrix} 
\hat\nuI_1\\ \hat\nuI_3\\ \vdots\\ \hat\nuI_{2g-1}\\
\end{pmatrix}.
$$

\item
$$
   \hat \iota_H \varpi^* \nuI
= \varpi^*  \iota_H  \nuI
=   \hat \iota_A \varpi^* \nuI.
$$
\item By defining
$$
   \left( 
       \sum_{i}^g\int_{\Gamma_{(x_i,y_i),-\infty}} \varpi^*\nuI \right),
$$
$\hat w_{\varpi^* \nuI} : S^g \Gamma_{-\infty}\hat C_{2g-1} \to \CC^g$
is a surjection.
\end{enumerate}
\end{proposition}

Figure 5 shows that as  half of $\beta_1$
consists of the path from $\infty$ to $B_1$, 
the path from $\pm\infty$ to $(0,0)$ in $\hat C_{2g-1}$
corresponds to a quarter of  $\hat \beta_1$.
Each $\hat \beta_a^\pm$ $(a =2, \ldots, g)$ consists of 
a contour from $\pm\infty$ to $\hat B_{2a -1}^\pm$.
Similarly we have $\hat \alpha_a^\pm$ $(a =1, \ldots, g)$.

Noting that $({B_1} \to B_2)$ lifts  to $(0,0) \to \hat B_2^{(\pm)})$,
we find that 
$$
 \int^{B_2}_{B_1} \nu^I = \frac{1}{2} \int^{\hat B_2}_{(0,0)} \varpi^*\nu^I
$$
 is a half-period in $\hat C_{2g-1}$.
The $(2(2g-1) \times g)$ matrix 
  $(\hat \omega', \hat \omega'')|_{\varpi \nuI}$ is given by
$$ 
(\omega'_1, \omega_2', \omega_2', \ldots,  \omega_g', \omega_g',
2\omega''_1, \omega_2'', \omega_2'', \ldots,  \omega_g'', \omega_g'').
$$
The corresponding lattice is denoted by $\hat \Lambda$ and the Jacobian
by $\hat \JJ = \CC^{2g-1}/\hat \Lambda$.

\begin{proposition} \label{prop:Jal}
Let $$
{\hat\JJ}^{\al,(1)}_g := 
\frac{\hat w_{\varpi^* \nuI}(S^g\Gamma_{-\infty} \hat C_{2g-1}) }
{\hat \Lambda \cap
\hat w_{\varpi^* \nuI}(S^g \Gamma_{-\infty}\hat C_{2g-1}) }.
$$
Then the following function is defined on ${\hat\JJ}^{\al,(1)}_g$,
$$
(z_1 z_2 \cdots z_g)(u),
$$ 
where 
$(z_1, z_2, \ldots, z_g)$
in $S^g\Gamma_{-\infty} \hat C_{2g-1}$ is any preimage of $u$ 
under the extended Abel map.

By identifying $w(S^g\Gamma_{\infty} C_g)=\CC^g$ and 
$\hat w_{\varpi^* \nuI}(S^g\Gamma_{-\infty}\hat C_{2g-1})=\CC^g$,
${\hat\JJ}^{\al,(1)}_g$ and ${\JJ}^{(1)}_g$ agree,
and their $\al_1$ function is expressed by
$$
\al_1(u) = (z_1 z_2 \cdots z_g)(u).
$$

\end{proposition}

\bigskip

\bigskip

\noindent
Shigeki Matsutani\\
\noindent
8-21-1 Higashi-Linkan Minami-ku,\\
\noindent
Sagamihara 252-0311,\\
\noindent
JAPAN.\\
\noindent
{rxb01142@nifty.com}\\

\bigskip

\noindent
Emma Previato\\
\noindent
Department of Mathematics and Statistics,\\
\noindent
Boston University,\\
\noindent
Boston, MA 02215-2411,\\
\noindent
U.S.A.\\
\noindent
{ep@bu.edu}\\


\begin{thebibliography}{AAAAA}
\bibitem[ACGH]{ACGH}
  \by{E. Arbarello, M. Cornalba, P. A. Griffiths, and J. Harris}
  \book{Geometry of Algebraic Curves Volume I}
  \publ{Springer}
  \yr{1984}.


\bibitem[AHP]{AHP}
\by{M. R. Adams, J. Harnad, and E. Previato}
\paper{Isospectral Hamiltonian flows in finite and infinite dimensions. I. 
Generalized Moser systems and moment maps into loop algebras}
\jour{Comm. Math. Phys.}
\vol{117} \yr{1988} \pages{451-500}. 

\bibitem[Ba1]{Ba1}
  \by{H.F. Baker}
  \book {Abelian functions. Abel's theorem and the allied theory of 
theta functions} 
Reprint of the 1897 original. 
With a foreword by Igor Krichever. 
Cambridge Mathematical Library. Cambridge University Press, Cambridge, 1995.

\bibitem[Ba2]{Ba2}
  \by{H. F. Baker}
  \paper{On the hyperelliptic sigma functions}
  \jour{Amer. J. of Math.}
  \vol{XX}
  \yr{1898}
  \pages{ 301-384}.

\bibitem[Ba3]{Ba3}
  \by{H. F. Baker}
  \paper{On a system of differential equations
leading to periodic functions}
  \jour{Acta Math.}
  \vol{27}
  \yr{1903}
  \pages{135-156}.

\bibitem[BEL1]{BEL1}
  \by{V. M. Buchstaber,  V. Z.  Enolskii, and  D. V. Leykin}
  \paper{ Kleinian Functions,
    Hyperelliptic Jacobians and Applications}
  \jour {Reviews in Mathematics and Mathematical Physics (London)}
  \eds{Novikov, S. P. and Krichever, I. M.}
  \publ{Gordon and Breach} \publaddr{India} \yr{1997}
 \pages{1-125}.

\bibitem[BEL2]{BEL2}
  \by{V. M. Buchstaber, V. Z. Enolskii, and D. V. Leykin}
  \paper{Uniformization of Jacobi Varieties of Trigonal Curves and Nonlinear
  Differential Equation}
  \jour{Funct. Anal. Appl.} \vol{34} \yr{2000} \pages{159-171}.

\bibitem[BLE]{BLE}
  \by{V. M. Buchstaber, D. V. Leykin, and  V. Z. Enolskii}
  \paper{$\sigma$-function of $(n, s)$-curves}
  \jour{Russian Math. Surveys} \vol{54} \yr{1999} \pages{628-629}.

\bibitem[C]{cornalba}
M. Cornalba, 
\paper{On the locus of curves with automorphisms},
\jour{Ann. Mat. Pura Appl.}
\vol{149} \yr{1987} \pages{135-151}. 

\bibitem[EEL]{EEL}
   \by{J. C. Eilbeck, V. Z. Enolskii and D. V. Leykin}
   \paper{On the Kleinian construction of Abelian 
    functions of canonical algebraic curves}
\book{In Proceedings of the Conference SIDE III: 
Symmetries of Integrable Differences Equations, Saubadia, May 1998, 
CRM Proceedings and Lecture Notes},
    \yr{2000}
    \pages{121-138}

\bibitem[EEM\^OP1]{EEMOP1}
   \by{ J.C. Eilbeck, V.Z. Enolskii, S. Matsutani,
Y. \^Onishi and E. Previato}
\paper{Addition formulae over the Jacobian pre-image of
hyperelliptic Wirtinger varieties }
   \jour{J. reine angew. Math.} {619} (2008), 37-48.

\bibitem[EEM\^OP2]{EEMOP2}
   \by{ J.C. Eilbeck, V.Z. Enolskii, S. Matsutani,
Y. \^Onishi and E. Previato}
   \paper{Abelian Functions for Trigonal Curves of Genus Three}
\jour{Int. Math. Research Notices} 
\vol{2007} \yr{2007} 140, \pages{1-38}.

\bibitem[EEP]{EEP}
   \by{J. C. Eilbeck, V. Z. Enolskii and E. Previato}
   \paper{Spectral Curves of Operators with Elliptic Coefficients}
   \jour{SIGMA} \vol{3} \yr{2007} 045 (17 pages). 

\bibitem[F]{F1}
  \by{J. D. Fay}
  \book{Theta functions on Riemann Surfaces}
  \publ{Springer}
  \yr{1973}.


\bibitem[Kl]{Kl1}
  \by{F. Klein}
  \paper{Ueber hyperelliptische Sigmafunctionen}
\jour{Math. Ann.} \vol{27}  \yr{1886} \pages{431-464}.



\bibitem[KMP]{KMP}
   \by{J. Komeda, S. Matsutani and E. Previato}
   \paper{The sigma function for Weierstrass semigroups 
$\langle3,7,8\rangle$ and $\langle6,13,14,15,16\rangle$}
\jour{Int. J. Math.}
\vol{24} \yr{2013} 1350085 (58pages).



\bibitem[LP]{lindqvistpeetre}
\by{P. Lindqvist and J. Peetre}
\paper{Two remarkable identities, called twos, for 
inverses to some Abelian integrals},
\jour{Amer. Math. Monthly} \vol{108}
 \yr{2001} \pages{403-410}. 

\bibitem[Ma0]{Ma0}
   \by{S. Matsutani}
    \paper{Hyperelliptic solutions of KdV and KP equations: 
reevaluation of Baker's study on hyperelliptic sigma functions}
\jour{J. Phys. A: Math. \& Gen.} 
\vol{34} \yr{2001} \pages{4721-4732}.


\bibitem[Ma1]{Ma01}
   \by{S. Matsutani}
\paper{Hyperelliptic solutions of modified Korteweg-de Vries
 equation of genus g: essentials of Miura transformation}
\jour{J. Phys. A: Math. \& Gen.} \vol{35}  (2002) 4321-4333.

\bibitem[Ma2]{Ma1}
   \by{S. Matsutani}
    \paper{On a relation of Weierstrass al-functions}
       \jour{Int.~J.~Appl.~Math.} \vol{11} (2002) 295-307.

\bibitem[Ma3]{Ma2}
   \by{S. Matsutani}
\paper{Hyperelliptic al Function
Solutions of sine-Gordon equation}
in
\lq\lq{New developments in mathematical physics research }
2004
Nova Science edited by V. Benton, 177-200.

\bibitem[Ma4]{Ma3}
   \by{S. Matsutani}
\paper{Neumann system and hyperelliptic al functions}
   \jour{Surv. Math. Appl.}  \vol{3} \yr{2008} 13-25.

\bibitem[MP]{MP}
   \by{S. Matsutani and E. Previato}
   \paper{Jacobi inversion on strata of the Jacobian of the 
$C_{rs}$ curve $y^r = f(x)$} 
\jour{J. Math. Soc. Japan}
\vol{60} \yr{2008} \pages{1009-1044}.

\bibitem[Mum]{Mum}
 \by{D. Mumford} \book{Tata Lectures on Theta} Vol.s I, II
Birkh\"auser  1981, 1984.


\bibitem[\^O]{O}
  \by{Y. \^Onishi}
  \paper{Determinant expressions in Abelian functions for
    purely trigonal curves of degree four}
  \jour{Int. J. Math.} \vol{20} \yr{2009} \pages{427-441}.

  \bibitem[P]{P}
  \by{E. Previato}
  \paper{Generalized Weierstrass $\mathfrak{p}$-functions
   and KP flows in affine space}
  \jour{Comment. Math. Helvetici} \vol{62} \yr{1987} \pages{292-310}.

  \bibitem[S]{S}
  \by{R. J. Schilling}
  \paper{Generalizations of the Neumann system:
a curve-theoretical approach--Part I, II, III order $n$ systems} 
  \jour{Comm. Pure Appl. Math} 
\vol{XL} \yr{1987} \pages{455-522}, 
\vol{XLII} \yr{1989} \pages{409-442}, 
\vol{XLV} \yr{1992} \pages{775-820}. 

\bibitem[Wei]{Wei}
\by{K. Weierstrass}
\paper{Zur Theorie der Abel'schen Functionen}
\jour{J. reine angew. Math.}
\vol{47} \yr{1854} \pages{289-306}. 

\bibitem[Wel]{Wel}
\by{J. Wellstein}
\paper{Zur Theorie der Functionenclasse 
$s^3 = (z-\alpha_1)(z-\alpha_2)\cdots(z-\alpha_6)$}
\jour{Math. Ann.}
\vol{52} (1898) \pages{440-448}.

\end{thebibliography}
\end{document}